\documentclass{imanum}
\usepackage{graphicx}

\jno{drnxxx}

\begin{document}

\title{WSLD operators: A class of fourth order difference approximations for space Riemann-Liouville derivative}

\author{%
{\sc
Minghua Chen
and
Weihua Deng\thanks{Corresponding author. Email: dengwh@lzu.edu.cn}} \\[2pt]
School of Mathematics and Statistics,
Lanzhou University, Lanzhou 730000, P. R. China
}

\maketitle

\begin{abstract}
{Because of the nonlocal properties of fractional operators, higher order schemes play more important role in discretizing fractional derivatives than classical ones. The striking feature is that higher order schemes of fractional derivatives can keep the same computation cost with first-order schemes but greatly improve the accuracy.
Nowadays, there are already two types of second order discretization schemes for space fractional derivatives: the first type is given and discussed  in [Sousa  \& Li, arXiv:1109.2345; Chen \& Deng, arXiv:1304.3788; Chen et al., Appl. Numer. Math., 70, 22-41]; and the second type is a class of schemes presented in [Tian et al., arXiv:1201.5949]. The core object of this paper is to derive a class of fourth order approximations, called the weighted and shifted Lubich difference (WSLD) operators, for space fractional derivatives. Then we use the derived schemes to solve the space fractional diffusion equation with variable coefficients in one-dimensional and two-dimensional cases. And the unconditional stability and the convergence with the global truncation error $\mathcal{O}(\tau^2+h^4)$  are theoretically proved and numerically verified.
}
{Fractional  diffusion equation; Weighted and shifted Lubich difference operators;  Numerical stability;  Convergence
}
\end{abstract}

\section{Introduction}
\label{sec;introduction}
In recent decades, fractional operators have been playing more and more important roles [\cite{Diethelm:10}], e.g., in mechanics (theory of viscoelasticity  and viscoplasticity), (bio-)chemistry (modelling of polymers and proteins), electrical engineering (transmission of ultrasound waves), medicine (modelling of human tissue under mechanical loads), etc. Efficiently solving the fractional partial differential equations (PDEs) naturally becomes an urgent topic. Because of the nonlocal properties of fractional operators, obtaining the analytical solutions of the fractional PDEs is more challenge or sometimes even impossible; or the obtained analytical solutions are less valuable (expressed by transcendental functions or infinite series). Luckily, some important progress has been made for numerically solving the fractional PDEs by finite difference methods, e.g., see [\cite{Meerschaert:04,Sousa:12,Sun:06,Tian:12,Yuste:06,Zhuang:09}].

In solving space fractional PDEs, high order finite difference schemes display more striking benefits because most of the time they can use the same computational cost with first order scheme but greatly improve the accuracy. For example, comparing with first order difference scheme which may have the matrix algebraic equation $(I-A)u^{n+1}=u^n+b^{n+1}$, the high order scheme has the matrix algebraic equation $(I-{\widetilde A})u^{n+1}=(I+{\widetilde B})u^n+{\tilde b}^{n+1/2}$. The three matrices $A$, ${\widetilde A}$ and ${\widetilde B}$ are all Toeplitz-like and have completely same structure, and the computational count for matrix vector multiplication is $\mathcal{O}(N \mbox{log} N)$, then the computational costs for solving the two matrix algebraic equations are almost the same [\cite{Deng:13}].

Nowadays, we notice that there exist two types of second order discretization schemes for space fractional derivatives. The idea of the first type is to combine the centered difference scheme of second classical derivative with piecewise linear polynomial approximation of the fractional integral.
\cite{Sousa:12} firstly use the idea to obtain the second order approximation in infinite domain. The paper [\cite{Chen:11}] detailedly analyzes the effectiveness of the approximation in finite domain. And this discretization is also effectively used to solve the time-space Capuo-Riesz fractional diffusion equation [\citet{Chen:13}]. The second type of second order approximation is in fact a class of second order discretization, which are obtained by assembling the Gr\"{u}nwald difference operators with different weights and shifts. This class of approximations are detailedly discussed and successfully applied to solve space fractional diffusion equations in [\cite{Tian:12}], and called WSGD operators there. Both of the two types of the operators have completely same structure, and the real parts of the eigenvalues of the matrixes are less than $0$, see [\cite{Deng:11, Tian:12}]. So they can be efficiently used to solve space fractional PDEs.

Based on Lubich's operator [\cite{Lubich:86}], this paper derives a class of fourth order approximations for space fractional derivatives, termed the weighted and shifted Lubich difference operators (WSLD operators). Using the  fractional linear multistep methods,
\cite{Lubich:86} obtains the $L$-th order ($L\leq 6$)  approximations of the $\alpha$-th derivative ($\alpha>0$) or integral ($\alpha<0$) by the corresponding coefficients of the generating functions $\delta ^{\alpha}(\zeta)$, where
\begin{equation}\label{1.1}
\delta^{\alpha}(\zeta) = \left(\sum_{i=1}^L\frac{1}{i}(1-\zeta)^i\right)^{\alpha}.
\end{equation}
 For $\alpha=1$, the scheme reduces to the classical $(L+1)$-point backward difference formula  [\cite{Henrici:62}]. For $L=1$, the scheme (\ref{1.1}) corresponds to the standard Gr\"{u}nwald discretization of $\alpha$-th derivative with first order accuracy;
unfortunately, for the time dependent equations the difference scheme
is unstable. But \cite{Meerschaert:04} successfully circumvent this difficulties by the so-called shifted Gr\"{u}wald formulae.
Taking $L=2$,  \cite{Cuesta:06}  discuss the convolution quadrature time discretization of fractional diffusion-wave equations;
when applying the discretization scheme to space fractional operator with $\alpha \in(1,2)$ for time dependent problem, the obtained scheme is also unstable,
since the eigenvalues of the matrix corresponding to the discretized operator are greater than one.  If using the shifted Lubich's formula, it reduces to the first order accuracy  (detailed description is given in Section 2). This paper weights and shifts Lubich's operator to obtain a class of fourth order discretization schemes, which are effective for time dependent problem.
Then we use the fourth order schemes to solve the following two-dimensional fractional diffusion equation with variable coefficients,
\begin{equation} \label{1.2}
\left\{ \begin{split}
\frac{\partial u(x,y,t) }{\partial t}&=d_{+}(x,y)\,_{x_L}D_x^{\alpha}u(x,y,t)+d_{-}(x,y)\, _{x}D_{x_R}^{\alpha}   u(x,y,t) \\
                                     &\quad + e_{+}(x,y)\,_{y_L}D_y^{\beta}u(x,y,t)+ e_{-}(x,y)\,  _{y}D_{y_R}^{\beta}u(x,y,t)+f(x,y,t),\\
                            u(x,y,0) &=u_0(x,y), ~~~~\, {\rm for}~~~ (x,y) \in \Omega,\\
                             u(x,y,t)&=0, ~~~~~~~~~~\,~~~~ {\rm for}~~~ (x,y,t) \in \partial \Omega \times (0,T],
 \end{split}
 \right.
\end{equation}
in the domain $\Omega=(x_L,x_R) \times (y_L,y_R)$, $0< t \leq T$,
where the orders of the  fractional derivatives are  $1<\alpha,\beta< 2$ and
$f(x,y,t)$ is a source term, and all the variable coefficients are nonnegative.
The left and right Riemann-Liouville  fractional
derivatives of the function $u(x)$ on $[x_L,x_R]$, $-\infty \leq x_L <x_R \leq \infty$ are, respectively, defined by [\cite{Podlubny:99,Samko:93}]
\begin{equation}\label{1.3}
 _{x_L}D_x^{\alpha}u(x)=
\frac{1}{\Gamma(2-\alpha)} \displaystyle \frac{\partial^2}{\partial x^2}
 \int_{x_L}\nolimits^x{\left(x-\xi\right)^{1-\alpha}}{u(\xi)}d\xi,
\end{equation}
and
\begin{equation}\label{1.4}
 _{x}D_{x_R}^{\alpha}u(x)=
 \frac{1}{\Gamma(2-\alpha)}\frac{\partial^2}{\partial x^2}
\int_{x}\nolimits^{x_R}{\left(\xi-x\right)^{1-\alpha}}{u(\xi)}d\xi.
\end{equation}

The  outline of this paper is as follows.
In Section 2, we derive a class of fourth order approximations for space fractional Riemann-Liouville derivatives, being
 effective in solving space fractional PDEs. In Section 3, the full discretization schemes of one-dimensional case of (\ref{1.2}) and  (\ref{1.2}) itself are presented. Section 4 does the detailed theoretical analyses for the stability  and convergence
of the given schemes. To show the effectiveness of the algorithm, we
perform the numerical experiments to verify the theoretical results
 in Section 5. Finally, we
conclude the paper with some remarks in the last section.

\section{Derivation of a class of fourth order discretizations for space fractional operators}\label{sec:1}
In the following, we derive a class of fourth order  approximations for Riemann-Liouville fractional derivatives, and prove that they are effective in solving space fractional PDE, i.e., all the eigenvalues of the matrixes corresponding to the discretized operators have negative real parts.

\subsection{Derivation of the discretization scheme}


Taking $L=2$, for all $|\zeta| \leq 1$, Eq. (\ref{1.1}) can be recast as
\begin{equation}\label{2.1}
\begin{split}
\left(\frac{3}{2}-2\zeta+\frac{1}{2}\zeta^2\right)^{\alpha} & =\left(\frac{3}{2}\right)^{\alpha}(1-\zeta)^{\alpha}(1-\frac{1}{3}\zeta)^{\alpha}\\
&=\left(\frac{3}{2}\right)^{\alpha} \sum_{n=0}^{\infty}(-1)^n \left( \begin{matrix} \alpha \\n \end{matrix} \right )\zeta^n
\cdot \sum_{m=0}^{\infty} \left(-\frac{1}{3}\right)^m  \left( \begin{matrix} \alpha \\m \end{matrix} \right )\zeta^m  \\
&=\left(\frac{3}{2}\right)^{\alpha} \sum_{n=0}^{\infty}\left[\sum_{m=0}^{\infty} (-1)^n \left( \begin{matrix} \alpha \\n \end{matrix} \right )
\cdot \left(-\frac{1}{3}\right)^m  \left( \begin{matrix} \alpha \\m \end{matrix} \right )\right]\zeta^{m+n}  \\
&= \sum_{k=0}^{\infty}q_k^{\alpha}\zeta^{k},
\end{split}
\end{equation}
with $k=m+n$,  and
\begin{equation}\label{2.2}
 q_k^{\alpha}=(-1)^k \left(\frac{3}{2}\right)^{\alpha} \sum_{m=0}^{k} 3^{-m} \left( \begin{matrix} \alpha \\k-m \end{matrix} \right )
  \left( \begin{matrix} \alpha \\m \end{matrix} \right )
  =\left(\frac{3}{2}\right)^{\alpha} \sum_{m=0}^{k} 3^{-m}g_m^{\alpha}g_{k-m}^{\alpha},
\end{equation}
where $g_k^{\alpha}=(-1)^k\left ( \begin{matrix}\alpha \\ k\end{matrix} \right )$ are the coefficients of the power series of the generating function $(1-\zeta)^{\alpha}$,
 and they can be calculated by the following recursively formula
\begin{equation}\label{2.3}
  g_0^{\alpha}=1, ~~~~g_k^{\alpha}=\left(1-\frac{\alpha+1}{k}\right)g_{k-1}^{\alpha},~~k \geq 1.
\end{equation}
If $\alpha<0$, $ \{q_k^{\alpha}\}_{k=0}^{\infty}$ correspond to the coefficients of  the 2nd order convolution quadrature for the approximation of fractional integral operator
 [see,~\cite{Cuesta:06}].

\begin{lemma}\label{lemma2.1}
 The coefficients in (\ref{2.2}) with $\alpha \in (1,2)$ satisfy the following properties
 \begin{equation*}
\begin{split}
&q_0^{\alpha}=\left(\frac{3}{2}\right)^{\alpha}>0; \!\!\quad
~~~~~~~~~~~~~~~~~~~~~~~~~~~~~~~~~~~~~~~~~~~~~~~~~~~~~~~~~~~~~~~~~~~~q_1^{\alpha}=-\left(\frac{3}{2}\right)^{\alpha}\frac{4\alpha}{3}<0; ~~\quad\\
&q_2^{\alpha}=\left(\frac{3}{2}\right)^{\alpha}\frac{\alpha(8\alpha-5)}{9}>0;
~~~\,~~~~~~~~~~~~~~~~~~~~~~~~~~~~~~~~~~~~~~~~~~~~~~~~~~q_3^{\alpha}=\left(\frac{3}{2}\right)^{\alpha}\frac{4\alpha(\alpha-1)(7-8\alpha)}{81}<0; \\
&q_4^{\alpha}=\left(\frac{3}{2}\right)^{\alpha}\frac{\alpha(\alpha-1)(64\alpha^2-176\alpha+123)}{486}>0;\\
&q_5^{\alpha}=\left(\frac{3}{2}\right)^{\alpha}\frac{2\alpha(\alpha-1)(2-\alpha)(64\alpha^2-208\alpha+183)}{3645}>0;
 ~~~~~~~\sum_{k=0}^{\infty}q_k^{\alpha}=0.\\
\end{split}
\end{equation*}
 \end{lemma}
\begin{proof}
Taking $\zeta=1$, it is easy to check that
  \begin{equation*}\label{2.5}
\begin{split}
\sum_{k=0}^{\infty}q_k^{\alpha}=\sum_{k=0}^{\infty}q_k^{\alpha}\zeta^{k}=\left(\frac{3}{2}-2\zeta+\frac{1}{2}\zeta^2\right)^{\alpha}=0.
\end{split}
\end{equation*}
\end{proof}

We first introduce two lemmas, which will be used to prove that the several classes of derived discretization schemes are 2nd, 3rd, and 4th order convergent, respectively.
\begin{lemma} [\cite{Ervin:06}] \label{lemma2.2}
 Let $\alpha>0$, $u \in C_0^{\infty}(\Omega)$, $\Omega \subset \mathbb{R}$, then
 \begin{equation*}
 \mathcal{F}(_{-\infty}D_x^{\alpha}u(x))=(-i\omega)^{\alpha}\widehat{u}(\omega) ~~ {\rm and}~~
  \mathcal{F}(_xD_{\infty}^{\alpha}u(x))=(i\omega)^{\alpha}\widehat{u}(\omega),
 \end{equation*}
 where $\mathcal{F}$ denotes the Fourier transform operator and $\widehat{u}(\omega)=\mathcal{F}(u)$, i.e.,
 \begin{equation*}
    \widehat{u}(\omega)=\int_{\mathbb{R}}e^{i\omega x }u(x)dx.
 \end{equation*}
\end{lemma}

\begin{lemma}\label{lemma2.3}
 Let $u $, $_{-\infty}D_x^{\alpha+1}u(x)$ (or $_{-\infty}D_x^{\alpha+2}u(x)$) with $\alpha \in (1,2)$  and their Fourier transforms belong to $L_1(\mathbb{R})$ when $p\neq 0$ (or $p=0$); and denote that
 \begin{equation}\label{2.4}
_LA_{p}^{\alpha}u(x)=\frac{1}{h^{\alpha}}\sum_{k=0}^{\infty}q_k^{\alpha}u(x-(k-p)h),
\end{equation}
where $q_k^{\alpha}$ is defined by (\ref{2.2}) and $p$ an integer. Then
$$_{-\infty}D_x^{\alpha}u(x)=\,_LA_{p}^{\alpha}u(x)+\mathcal{O}(h),   \,~~~~p\neq 0,$$
and
$$  _{-\infty}D_x^{\alpha}u(x)=\,_LA_{p}^{\alpha}u(x)+\mathcal{O}(h^2),   ~~~p=0.$$

\end{lemma}
\begin{proof}
From (\ref{2.2}) and $k=m+n$, we obtain
\begin{equation*}
\begin{split}
\mathcal{F}(_LA_{p}^{\alpha}u)(\omega)&=h^{-\alpha}\sum_{k=0}^{\infty}q_k^{\alpha} \mathcal{F}\left(u(x-(k-p)h)\right)(\omega) \\
&=h^{-\alpha}e^{-i\omega ph} \sum_{k=0}^{\infty}q_k^{\alpha} \left(e^{i\omega h}\right)^k \widehat{u}(\omega)\\
&=h^{-\alpha}e^{-i\omega ph}\left(\frac{3}{2}\right)^{\alpha} \sum_{n=0}^{\infty}(-1)^n \left( \begin{matrix} \alpha \\n \end{matrix} \right )e^{i\omega nh}
\cdot \sum_{m=0}^{\infty} \left(-\frac{1}{3}\right)^m  \left( \begin{matrix} \alpha \\m \end{matrix} \right )e^{i\omega mh} \widehat{u}(\omega)\\
&=(-i\omega)^{\alpha}\left[e^{-i\omega ph} \left(\frac{1-e^{i\omega h}}{-i\omega h}\right )^{\alpha} \right]
\left( 1+\frac{1}{2}\left(1-e^{i\omega h}\right)\right)^{\alpha} \widehat{u}(\omega)\\
&=(-i\omega)^{\alpha}e^{ pz} \left(\frac{1-e^{-z}}{z}\right )^{\alpha}
\left( 1+\frac{1}{2}\left(1-e^{-z}\right)\right)^{\alpha} \widehat{u}(\omega),
  \end{split}
\end{equation*}
with $z=-i\omega h$. It is easy to check that
 \begin{equation*}
\begin{split}
e^{ pz} \left(\frac{1-e^{-z}}{z}\right )^{\alpha}
&=\Big[1+\Big(p-\frac{\alpha}{2}\Big)z+ \Big(\frac{1}{2}p^2 -\frac{\alpha }{2}p+\frac{3\alpha^2+\alpha}{24}\Big) z^2 \\
&\quad +\Big(\frac{1}{6}p^3 -\frac{\alpha }{4}p^2
+\frac{3\alpha^2+\alpha}{24}p-\frac{\alpha^3+\alpha^2}{48} \Big)z^3  +\mathcal{O}(z^4)    \Big],
  \end{split}
\end{equation*}
and
 \begin{equation*}
\begin{split}
&\left( 1+\frac{1}{2}\left(1-e^{-z}\right)\right)^{\alpha}
=\Big[1+ \frac{\alpha}{2}z+\frac{\alpha(\alpha-3)}{8}z^2+\frac{\alpha(\alpha^2-9\alpha+12)}{48}z^3+\mathcal{O}(z^4)\Big],
  \end{split}
\end{equation*}
then we have
\begin{equation}\label{2.5}
\begin{split}
&e^{ pz} \left(\frac{1-e^{-z}}{z}\right )^{\alpha} \!\! \left( 1+\frac{1}{2}\left(1-e^{-z}\right)\right)^{\alpha}
\!\!  =1+pz+\frac{3p^2-2\alpha}{6}z^2+\frac{2p^3+\alpha(3-4p)}{12}z^3+\mathcal{O}(z^4).
  \end{split}
\end{equation}
Therefore, from Lemma \ref{lemma2.1}, we get
\begin{equation*}
\begin{split}
\mathcal{F}(_LA_{p}^{\alpha}u)(\omega)=\mathcal{F}(_{-\infty}D_x^{\alpha}u(x))+ \widehat{\phi}(\omega),
  \end{split}
\end{equation*}
where $ \widehat{\phi}(\omega)=(-i\omega)^{\alpha}\left(pz+\frac{3p^2-2\alpha}{6}z^2+\frac{2p^3+\alpha(3-4p)}{12}z^3+\mathcal{O}(z^4)\right)\widehat{u}(\omega)$.
Then there exists
\begin{equation*}
\begin{split}
&|\widehat{\phi}(\omega)| \leq \widetilde{c}|i\omega|^{\alpha+1}|\widehat{u}(\omega)|\cdot h, ~~~~~~p\neq 0,\\
&|\widehat{\phi}(\omega)| \leq c|i\omega|^{\alpha+2}|\widehat{u}(\omega)|\cdot h^2, ~\,~~~p=0.
\end{split}
\end{equation*}
Hence
\begin{equation*}
\begin{split}
|_{-\infty}D_x^{\alpha}u(x)-\,_LA_{p}^{\alpha}u(x)|=|\phi(x)| \leq \frac{1}{2\pi}\int_{\mathbb{R}}|\widehat{\phi}(\omega)|dx= \left\{ \begin{array}
 {l@{\quad} l}
\mathcal{O}(h),&p\neq 0,\\
\mathcal{O}(h^2),& p=0.
 \end{array}
 \right.
  \end{split}
\end{equation*}
\end{proof}

In the following, we present the approximation operators for Riemann-Liouville derivative and prove that they have 2nd, 3rd, and 4th order truncation errors.
\begin{theorem}\label{theorem2.1}(Second order approximations for left Riemann-Liouville  derivative)
 Let $u$, $_{-\infty}D_x^{\alpha+2}u(x)$ with $\alpha \in (1,2)$ and their Fourier transforms  belong to $L_1(\mathbb{R})$. Denote that
 \begin{equation}\label{2.6}
_{2L}A_{p,q}^{\alpha}u(x)=w_p\,_LA_{p}^{\alpha}u(x)+w_q\,_LA_{q}^{\alpha}u(x),
\end{equation}
where $_LA_{p}^{\alpha}$, $_LA_{q}^{\alpha}$  are defined by (\ref{2.4}), $w_p=\frac{q}{q-p}$,  $w_q=\frac{p}{p-q}$, $p \neq  q$,
 and $p$, $q$ are integers. Then
\begin{equation*}
\begin{split}
  &_{-\infty}D_x^{\alpha}u(x)=\,_{2L}A_{p,q}^{\alpha}u(x)+\mathcal{O}(h^2).
 \end{split}
\end{equation*}
\end{theorem}
\begin{proof}
 From the proof of Lemma  \ref{lemma2.3}, we have
  \begin{equation*}
\begin{split}
\mathcal{F}(_LA_{p}^{\alpha}u)(\omega)
&=(-i\omega)^{\alpha}\Big[ 1+pz+\frac{3p^2-2\alpha}{6}z^2+\frac{2p^3+\alpha(3-4p)}{12}z^3+\mathcal{O}(z^4) \Big] \widehat{u}(\omega)
  \end{split}
\end{equation*}
and
  \begin{equation*}
\begin{split}
\mathcal{F}(_LA_{q}^{\alpha}u)(\omega)
&=(-i\omega)^{\alpha}\Big[ 1+qz+\frac{3q^2-2\alpha}{6}z^2+\frac{2q^3+\alpha(3-4q)}{12}z^3+\mathcal{O}(z^4) \Big] \widehat{u}(\omega).
  \end{split}
\end{equation*}
Then there exists
 \begin{equation*}
\begin{split}
\mathcal{F}(_{2L}A_{p,q}^{\alpha}u)(\omega)
=(-i\omega)^{\alpha}\Big[ 1-\frac{3pq+2\alpha}{6}z^2-\frac{2pq(p+q)-3\alpha}{12}z^3+\mathcal{O}(z^4) \Big] \widehat{u}(\omega),
  \end{split}
\end{equation*}
and by the similar way to the proof of Lemma  \ref{lemma2.3} we get
\begin{equation*}
\begin{split}
  &_{-\infty}D_x^{\alpha}u(x)=\,_{2L}A_{p,q}^{\alpha}u(x)+\mathcal{O}(h^2).
 \end{split}
\end{equation*}
\end{proof}
\begin{theorem}\label{theorem2.2}(Third order approximations for left Riemann-Liouville  derivative)
 Let $u$, $_{-\infty}D_x^{\alpha+3}u(x)$ with $\alpha \in (1,2)$ and their Fourier transforms  belong to $L_1(\mathbb{R})$. Denote that
 \begin{equation}\label{2.7}
_{3L}A_{p,q,r,s}^{\alpha}u(x)=w_{p,q}\,_{2L}A_{p,q}^{\alpha}u(x)+w_{r,s}\,_{2L}A_{r,s}^{\alpha}u(x),
\end{equation}
where $_{2L}A_{p,q}^{\alpha}$ and $_{2L}A_{r,s}^{\alpha}$  are defined by (\ref{2.6}),  $w_{p,q}=\frac{3rs+2\alpha}{3(rs-pq)}$,   $w_{r,s}=\frac{3pq+2\alpha}{3(pq-rs)}$, $rs \neq  pq$,
 and  $p$, $q$, $r$, $s$ are integers. Then
\begin{equation*}
\begin{split}
  &_{-\infty}D_x^{\alpha}u(x)=\,_{3L}A_{p,q,r,s}^{\alpha}u(x)+\mathcal{O}(h^3).
 \end{split}
\end{equation*}
\end{theorem}
\begin{proof}
 By the proof of Theorem  \ref{theorem2.1}, we have
 \begin{equation*}
\begin{split}
\mathcal{F}(_{2L}A_{p,q}^{\alpha}u)(\omega)
=(-i\omega)^{\alpha}\Big[ 1-\frac{3pq+2\alpha}{6}z^2-\frac{2pq(p+q)-3\alpha}{12}z^3+\mathcal{O}(z^4) \Big] \widehat{u}(\omega)
  \end{split}
\end{equation*}
and
 \begin{equation*}
\begin{split}
\mathcal{F}(_{2L}A_{r,s}^{\alpha}u)(\omega)
=(-i\omega)^{\alpha}\Big[ 1-\frac{3rs+2\alpha}{6}z^2-\frac{2rs(r+s)-3\alpha}{12}z^3+\mathcal{O}(z^4) \Big] \widehat{u}(\omega).
  \end{split}
\end{equation*}
Then there exists
 \begin{equation*}
\begin{split}
&\mathcal{F}(_{3L}A_{p,q,r,s}^{\alpha}u)(\omega)\\
&\quad =(-i\omega)^{\alpha}\Big[ 1+\frac{6pqrs(r+s-p-q)+4\alpha\big[rs(r+s)-pq(p+q)\big]+9\alpha(rs-pq)}{36(rs-pq)}z^3
+\mathcal{O}(z^4) \Big] \widehat{u}(\omega),
  \end{split}
\end{equation*}
and by the similar way to the proof of Lemma  \ref{lemma2.3} we get
\begin{equation*}
\begin{split}
  &_{-\infty}D_x^{\alpha}u(x)=\,_{3L}A_{p,q,r,s}^{\alpha}u(x)+\mathcal{O}(h^3).
 \end{split}
\end{equation*}
\end{proof}

\begin{theorem}\label{theorem2.3}(Fourth order approximations for left Riemann-Liouville  derivative)
 Let $u$, $_{-\infty}D_x^{\alpha+4}u(x)$ with $\alpha \in (1,2)$ and their Fourier transforms  belong to $L_1(\mathbb{R})$. Denote that
 \begin{equation}\label{2.8}
_{4L}A_{p,q,r,s,\overline{p},\overline{q},\overline{r},\overline{s}}^{\alpha}u(x)=w_{p,q,r,s}\,_{3L}A_{p,q,r,s}^{\alpha}u(x)
+w_{\overline{p},\overline{q},\overline{r},\overline{s}}\,_{3L}A_{\overline{p},\overline{q},\overline{r},\overline{s}}^{\alpha}u(x),
\end{equation}
where $_{3L}A_{p,q,r,s}^{\alpha}$ and $_{3L}A_{\overline{p},\overline{q},\overline{r},\overline{s}}^{\alpha}$  are defined by (\ref{2.7});
 and
\begin{equation}\label{2.9}
\begin{split}
w_{p,q,r,s}=\frac{a_{p,q,r,s}\,\overline{b}_{\overline{p},\overline{q},\overline{r},\overline{s}}}
{a_{p,q,r,s}\,\overline{b}_{\overline{p},\overline{q},\overline{r},\overline{s}}-\overline{a}_{\overline{p},\overline{q},\overline{r},\overline{s}}\,b_{p,q,r,s}};
 \end{split}
\end{equation}
\begin{equation}\label{2.10}
\begin{split}
w_{\overline{p},\overline{q},\overline{r},\overline{s}}
=\frac{\overline{a}_{\overline{p},\overline{q},\overline{r},\overline{s}}\,b_{p,q,r,s}}
{\overline{a}_{\overline{p},\overline{q},\overline{r},\overline{s}}\,b_{p,q,r,s}-a_{p,q,r,s}\,\overline{b}_{\overline{p},\overline{q},\overline{r},\overline{s}}};
 \end{split}
\end{equation}
with
\begin{equation*}
\begin{split}
a_{p,q,r,s}=rs-pq; \quad
b_{p,q,r,s}=6pqrs(r+s-p-q)+4\alpha\big[rs(r+s)-pq(p+q)\big]+9\alpha(rs-pq);
 \end{split}
\end{equation*}
\begin{equation*}
\begin{split}
\overline{a}_{\overline{p},\overline{q},\overline{r},\overline{s}}=\overline{r}\,\overline{s}-\overline{p}\,\overline{q}; \quad
\overline{b}_{\overline{p},\overline{q},\overline{r},\overline{s}}=
6\overline{p}\,\overline{q}\,\overline{r}\,\overline{s}(\overline{r}+\overline{s}-\overline{p}-\overline{q})
+4\alpha\big[\overline{r}\,\overline{s}(\overline{r}+\overline{s})-\overline{p}\,\overline{q}(\overline{p}+\overline{q})\big]
+9\alpha(\overline{r}\,\overline{s}-\overline{p}\,\overline{q});
 \end{split}
\end{equation*}
and
$a_{p,q,r,s}\,\overline{b}_{\overline{p},\overline{q},\overline{r},\overline{s}} \neq \overline{a}_{\overline{p},\overline{q},\overline{r},\overline{s}}\,b_{p,q,r,s}$; $p$, $q$, $r$, $s$;  $\overline{p}$, $\overline{q}$, $\overline{r}$, $\overline{s}$ are integers. Then
\begin{equation*}
\begin{split}
  _{-\infty}D_x^{\alpha}u(x)=\,_{4L}A_{p,q,r,s,\overline{p},\overline{q},\overline{r},\overline{s}}^{\alpha}u(x)+\mathcal{O}(h^4).
 \end{split}
\end{equation*}
\end{theorem}
\begin{proof}
According to the proof of Theorem  \ref{theorem2.2}, we have
 \begin{equation*}
\begin{split}
&\mathcal{F}(_{3L}A_{p,q,r,s}^{\alpha}u)(\omega)\\
&\quad =(-i\omega)^{\alpha}\Big[ 1+\frac{6pqrs(r+s-p-q)+4\alpha\big[rs(r+s)-pq(p+q)\big]+9\alpha(rs-pq)}{36(rs-pq)}z^3
+\mathcal{O}(z^4) \Big] \widehat{u}(\omega)
  \end{split}
\end{equation*}
and
 \begin{equation*}
\begin{split}
&\mathcal{F}(_{3L}A_{\overline{p},\overline{q},\overline{r},\overline{s}}^{\alpha}u)(\omega)\\
&\quad =(-i\omega)^{\alpha}\Big[ 1+\frac{6\overline{p}\,\overline{q}\,\overline{r}\,\overline{s}(\overline{r}+\overline{s}-\overline{p}-\overline{q})
+4\alpha\big[\overline{r}\,\overline{s}(\overline{r}+\overline{s})-\overline{p}\,\overline{q}(\overline{p}+\overline{q})\big]
+9\alpha(\overline{r}\,\overline{s}-\overline{p}\,\overline{q})}{36(\overline{r}\,\overline{s}-\overline{p}\,\overline{q})}z^3
+\mathcal{O}(z^4) \Big] \widehat{u}(\omega).
  \end{split}
\end{equation*}
Then there exists
 \begin{equation*}
\begin{split}
&\mathcal{F}(_{4L}A_{p,q,r,s,\overline{p},\overline{q},\overline{r},\overline{s}}^{\alpha}u)(\omega) =(-i\omega)^{\alpha}\left( 1+\mathcal{O}(z^4) \right) \widehat{u}(\omega),
  \end{split}
\end{equation*}
and by the similar way to the proof of Lemma  \ref{lemma2.3} we get
\begin{equation*}
\begin{split}
  _{-\infty}D_x^{\alpha}u(x)=\,_{4L}A_{p,q,r,s,\overline{p},\overline{q},\overline{r},\overline{s}}^{\alpha}u(x)+\mathcal{O}(h^4).
 \end{split}
\end{equation*}
\end{proof}

For the right Riemann-Liouville fractional derivative, denote that
 \begin{equation}\label{2.11}
_RA_{p}^{\alpha}u(x)=\frac{1}{h^{\alpha}}\sum_{k=0}^{\infty}q_k^{\alpha}u(x+(k-p)h),
\end{equation}
where $q_k^{\alpha}$ is defined by (\ref{2.2}) and $p$  an integer. Using the same way as Theorems \ref{theorem2.1}-\ref{theorem2.3}, we can obtain the following results. In particular, the coefficients in (\ref{2.12}) are completely the same as the ones in (\ref{2.6}); the coefficients in (\ref{2.13}) the same as the ones in (\ref{2.7}); and the coefficients in (\ref{2.14}) the same as the ones in (\ref{2.8}).

\begin{theorem}\label{theorem2.4}(Second order approximations for right Riemann-Liouville  derivative)
 Let $u$, $_{x}D_{\infty}^{\alpha+2}u(x)$ with $\alpha \in (1,2)$ and their Fourier transforms  belong to $L_1(\mathbb{R})$, and denote that
 \begin{equation}\label{2.12}
_{2R}A_{p,q}^{\alpha}u(x)=w_p\,_RA_{p}^{\alpha}u(x)+w_q\,_RA_{q}^{\alpha}u(x),
\end{equation}
then
\begin{equation*}
\begin{split}
  &_{x}D_{\infty}^{\alpha}u(x)=\,_{2R}A_{p,q}^{\alpha}u(x)+\mathcal{O}(h^2).
 \end{split}
\end{equation*}
\end{theorem}

\begin{theorem}\label{theorem2.5}(Third order approximations for right Riemann-Liouville  derivative)
 Let $u$, $_{x}D_{\infty}^{\alpha+3}u(x)$ with $\alpha \in (1,2)$ and their Fourier transforms  belong to $L_1(\mathbb{R})$, and denote that
 \begin{equation}\label{2.13}
_{3R}A_{p,q,r,s}^{\alpha}u(x)=w_{p,q}\,_{2R}A_{p,q}^{\alpha}u(x)+w_{r,s}\,_{2R}A_{r,s}^{\alpha}u(x),
\end{equation}
then
\begin{equation*}
\begin{split}
  &_{x}D_{\infty}^{\alpha}u(x)=\,_{3R}A_{p,q,r,s}^{\alpha}u(x)+\mathcal{O}(h^3).
 \end{split}
\end{equation*}
\end{theorem}

\begin{theorem}\label{theorem2.6}(Fourth order approximations for right Riemann-Liouville  derivative)
 Let $u$, $_{x}D_{\infty}^{\alpha+4}u(x)$ with $\alpha \in (1,2)$ and their Fourier transforms  belong to $L_1(\mathbb{R})$, and denote that
 \begin{equation}\label{2.14}
_{4R}A_{p,q,r,s,\overline{p},\overline{q},\overline{r},\overline{s}}^{\alpha}u(x)=w_{p,q,r,s}\,_{3R}A_{p,q,r,s}^{\alpha}u(x)
+w_{\overline{p},\overline{q},\overline{r},\overline{s}}\,_{3R}A_{\overline{p},\overline{q},\overline{r},\overline{s}}^{\alpha}u(x),
\end{equation}
then
\begin{equation*}
\begin{split}
  _{x}D_{\infty}^{\alpha}u(x)=\,_{4R}A_{p,q,r,s,\overline{p},\overline{q},\overline{r},\overline{s}}^{\alpha}u(x)+\mathcal{O}(h^4).
 \end{split}
\end{equation*}
\end{theorem}

All the above schemes are applicable to bounded domain, say, $(x_L,\,x_R)$, after performing zero extensions to the functions considered.  Let $u(x)$ be the zero extended function from the bounded domain $(x_L,\,x_R)$, and satisfy the requirements of the above corresponding theorems (Theorems \ref{theorem2.1}-\ref{theorem2.6}). Denoting
 \begin{equation}\label{2.15}
  \begin{split}
&_L\widetilde{A}_{p}^{\alpha}u(x)=\frac{1}{h^{\alpha}}\sum_{k=0}^{[\frac{x-x_L}{h}]+p}q_k^{\alpha}u(x-(k-p)h),
\end{split}
\end{equation}
 then
 \begin{equation}\label{2.16}
\begin{split}
_{x_L}D_{x}^{\alpha}u(x)&=\,_L\widetilde{A}_{p}^{\alpha}u(x)+\mathcal{O}(h), ~\,\quad p \neq 0,\\
                            _{x_L}D_{x}^{\alpha}u(x) &=\,_L\widetilde{A}_{p}^{\alpha}u(x)+\mathcal{O}(h^2), \quad p = 0;
 \end{split}
 \end{equation}
\begin{equation}\label{2.17}
_{x_L}D_{x}^{\alpha}u(x)=\,_{2L}\widetilde{A}_{p,q}^{\alpha}u(x)+\mathcal{O}(h^2),~~~~ {\rm where}~~  _{2L}\widetilde{A}_{p,q}^{\alpha}u(x)=w_p\,_L\widetilde{A}_{p}^{\alpha}u(x)+w_q\,_L\widetilde{A}_{q}^{\alpha}u(x);
 \end{equation}
\begin{equation}\label{2.18}
_{x_L}D_{x}^{\alpha}u(x)=\,_{3L}\widetilde{A}_{p,q,r,s}^{\alpha}u(x)+\mathcal{O}(h^3), ~~~~ {\rm where}~~  _{3L}\widetilde{A}_{p,q,r,s}^{\alpha}u(x)=w_{p,q}\,_{2L}\widetilde{A}_{p,q}^{\alpha} u(x)+w_{r,s}\,_{2L}\widetilde{A}_{p,q}^{\alpha} u(x);
\end{equation}
and
\begin{equation}\label{2.19}
_{x_L}D_{x}^{\alpha}u(x)=\,_{4L}\widetilde{A}_{p,q,r,s,\overline{p},\overline{q},\overline{r},\overline{s}}^{\alpha}u(x)+\mathcal{O}(h^4),
\end{equation}
where $ _{4L}\widetilde{A}_{p,q,r,s,\overline{p},\overline{q},\overline{r},\overline{s}}^{\alpha}u(x)
=w_{p,q,r,s}\,_{3L}\widetilde{A}_{p,q,r,s}^{\alpha} u(x)
+w_{\overline{p},\overline{q},\overline{r},\overline{s}}\,_{3L}\widetilde{A}_{\overline{p},\overline{q},\overline{r},\overline{s}}^{\alpha} u(x)$.

Denoting $x_i=x_L+ih$, $i=-m,\ldots, 0,1,\ldots,N_x-1,N_x,\ldots,N_x+m$,
and $h=(x_R-x_L)/N_x$ being the uniform space stepsize, it can be noted that
 \begin{equation*}
u(x_i)=0, \quad {\rm for}~~ i=-m,-m+1,\ldots,0 \quad {\rm and }\quad   i=N_x, N_x+1,\ldots, N_x+m,
 \end{equation*}
where
\begin{equation}\label{2.20}
  m=max(abs(p,q,r,s,\overline{p},\overline{q},\overline{r},\overline{s})).
\end{equation}
 Then the approximation operator of (\ref{2.15}) can be described as
 \begin{equation}\label{2.21}
  \begin{split}
&_L\widetilde{A}_{p}^{\alpha}u(x_i)=\frac{1}{h^{\alpha}}\sum_{k=0}^{i+p}q_k^{\alpha}u(x_{i-k+p})
=\frac{1}{h^{\alpha}}\sum_{k=m-p}^{i+m}q_{k+p-m}^{\alpha}u(x_{i-k+m})
=\frac{1}{h^{\alpha}}\sum_{k=0}^{i+m}q_{k+p-m}^{\alpha}u(x_{i-k+m}),
\end{split}
\end{equation}
where $q_{k+p-m}^{\alpha}=0$, when $k+p-m<0$, and $p$ is an integer.
Then
 \begin{equation}\label{2.22}
\begin{split}
_{x_L}D_{x}^{\alpha}u(x_i)&=\,_L\widetilde{A}_{p}^{\alpha}u(x_i)+\mathcal{O}(h)
=\frac{1}{h^{\alpha}}\sum_{k=0}^{i+m}q_{k+p-m}^{\alpha}u(x_{i-k+m})+\mathcal{O}(h),      ~\quad p \neq 0,\\
_{x_L}D_{x}^{\alpha}u(x_i)&=\,_L\widetilde{A}_{p}^{\alpha}u(x_i)+\mathcal{O}(h^2)
=\frac{1}{h^{\alpha}}\sum_{k=0}^{i+m}q_{k+p-m}^{\alpha}u(x_{i-k+m})+\mathcal{O}(h^2),     \quad p = 0;
 \end{split}
 \end{equation}
\begin{equation}\label{2.23}
 \begin{split}
_{x_L}D_{x}^{\alpha}u(x_i)&= \,_{2L}\widetilde{A}_{p,q}^{\alpha}u(x_i) +\mathcal{O}(h^2)
=\frac{1}{h^{\alpha}}\sum_{k=0}^{i+m}(w_pq_{k+p-m}^{\alpha}+w_qq_{k+q-m}^{\alpha})u(x_{i-k+m}) +\mathcal{O}(h^2);
 \end{split}
\end{equation}
\begin{equation}\label{2.24}
\begin{split}
_{x_L}D_{x}^{\alpha}u(x_i)&\!=\!\,_{3L}\widetilde{A}_{p,q,r,s}^{\alpha}u(x)+\mathcal{O}(h^3)\\
&\!=\!\frac{1}{h^{\alpha}}\!\!\sum_{k=0}^{i+m}(w_{p,q}w_pq_{k+p-m}^{\alpha}\!+\!w_{p,q}w_qq_{k+q-m}^{\alpha}
\!+\!w_{r,s}w_rq_{k+r-m}^{\alpha}\!+\!w_{r,s}w_sq_{k+s-m}^{\alpha})u(x_{i-k+m})\! \\
 & ~~~~ +\!\mathcal{O}(h^3);
 \end{split}
\end{equation}
\begin{equation}\label{2.25}
\begin{split}
&_{x_L}D_{x}^{\alpha}u(x_i)=\,_{4L}\widetilde{A}_{p,q,r,s,\overline{p},\overline{q},\overline{r},\overline{s}}^{\alpha}u(x_i)+\mathcal{O}(h^4)
=\frac{1}{h^{\alpha}}\sum_{k=0}^{i+m}\varphi_k^{\alpha}u(x_{i-k+m})+\mathcal{O}(h^4),
 \end{split}
\end{equation}
where
\begin{equation}\label{2.26}
\begin{split}
\varphi_k^{\alpha}&=w_{p,q,r,s}  w_{p,q}w_pq_{k+p-m}^{\alpha}+ w_{p,q,r,s}w_{p,q}w_qq_{k+q-m}^{\alpha}
+w_{p,q,r,s}w_{r,s}w_rq_{k+r-m}^{\alpha}\\
&\quad+w_{p,q,r,s}w_{r,s}w_sq_{k+s-m}^{\alpha} +w_{\overline{p},\overline{q},\overline{r},\overline{s}}w_{\overline{p},\overline{q}}w_{\overline{p}}q_{k+\overline{p}-m}^{\alpha}
+w_{\overline{p},\overline{q},\overline{r},\overline{s}}w_{\overline{p},\overline{q}}w_{\overline{q}}q_{k+\overline{q}-m}^{\alpha}
 \\
&
\quad+w_{\overline{p},\overline{q},\overline{r},\overline{s}}w_{\overline{r},\overline{s}}w_{\overline{r}}q_{k+\overline{r}-m}^{\alpha}+w_{\overline{p},\overline{q},\overline{r},\overline{s}}w_{\overline{r},\overline{s}}w_{\overline{s}}q_{k+\overline{s}-m}^{\alpha}.
\end{split}
\end{equation}
Taking $U=[u({x_1}),u({x_2}),\cdots,u({x_{N_x-1}})]^{\rm T}$,  then (\ref{2.21}) can be rewritten as the matrix form
\begin{equation}\label{2.27}
  \begin{split}
&_L\widetilde{A}_{p}^{\alpha}U=\frac{1}{h^{\alpha}}A_{p}^\alpha U,
\end{split}
\end{equation}
where
\begin{equation}\label{2.28}
A_p^{\alpha}=\left [ \begin{matrix}
q_p^\alpha      &q_{p-1}^\alpha   & \cdots         &q_0^\alpha     &             &                  &               &        \\
q_{p+1}^\alpha  &q_p^\alpha       &q_{p-1}^\alpha  &  \cdots       & q_0^\alpha  &                  &               &         \\
q_{p+2}^\alpha  &q_{p+1}^\alpha   &q_p^\alpha      &q_{p-1}^\alpha &\cdots       &    q_0^\alpha    &               &          \\
\vdots          &    \ddots       &\ddots          & \ddots        & \ddots      &   \cdots         &     \ddots    &           \\
q_{n-2}^\alpha  &\cdots           & \ddots         &q_{p+1}^\alpha &q_{p}^\alpha &  q_{p-1}^\alpha  & \cdots        &q_0^\alpha  \\
\vdots          &\ddots           & \cdots         & \ddots        & \ddots      &   \ddots         & \ddots        &\vdots       \\
q_{p+n-3}^\alpha&\cdots           &  \ddots        & \cdots        & \ddots      &  q_{p+1}^\alpha  &q_{p}^\alpha   &q_{p-1}^\alpha\\
q_{p+n-2}^\alpha&q_{p+n-3}^\alpha &  \cdots        &q_{n-2}^\alpha & \cdots      &    q_{p+2}^\alpha&q_{p+1}^\alpha &q_{p}^\alpha
 \end{matrix}
 \right ],
\end{equation}
and $p$ is an integer and $q_{k}^{\alpha}=0$, when $k<0$. From (\ref{2.23})-(\ref{2.25}) we obtain
\begin{equation}\label{2.29}
\begin{split}
&_{2L}\widetilde{A}_{p,q}^{\alpha}U=\frac{1}{h^{\alpha}}A_{p,q}^\alpha U, \quad
A_{p,q}^\alpha =w_p\,A_{p}^\alpha+w_q\,A_{q}^\alpha;
 \end{split}
\end{equation}
\begin{equation}\label{2.30}
\begin{split}
&_{3L}\widetilde{A}_{p,q,r,s}^{\alpha}U=\frac{1}{h^{\alpha}}A_{p,q,r,s}^\alpha U, \quad
A_{p,q,r,s}^\alpha =w_{p,q}\,A_{p,q}+w_{r,s}\,A_{r,s};
 \end{split}
\end{equation}
\begin{equation}\label{2.31}
\begin{split}
&_{4L}\widetilde{A}_{p,q,r,s,\overline{p},\overline{q},\overline{r},\overline{s}}^{\alpha}U
=\frac{1}{h^{\alpha}}A_{p,q,r,s,\overline{p},\overline{q},\overline{r},\overline{s}}^\alpha U, \quad
A_{p,q,r,s,\overline{p},\overline{q},\overline{r},\overline{s}}^\alpha =w_{p,q,r,s}\,A_{p,q,r,s}^\alpha
+w_{\overline{p},\overline{q},\overline{r},\overline{s}}\,A_{\overline{p},\overline{q},\overline{r},\overline{s}}^\alpha.
 \end{split}
\end{equation}
Similarly, for the right Riemann-Liouville derivative, taking
 \begin{equation*}
  \begin{split}
&_R\widetilde{A}_{p}^{\alpha}u(x)=\frac{1}{h^{\alpha}}\sum_{k=0}^{[\frac{x_R-x}{h}]+p}q_k^{\alpha}u(x+(k-p)h),
\end{split}
\end{equation*}
then there exists
\begin{equation*}
  \begin{split}
&_R\widetilde{A}_{p}^{\alpha}u(x_i)=\frac{1}{h^{\alpha}}\!\!\!\sum_{k=0}^{N_x-i+p}\!\!\!q_k^{\alpha}u(x_{i+k-p})
=\frac{1}{h^{\alpha}}\!\!\!\sum_{k=m-p}^{N_x-i+m}\!\!\!q_{k+p-m}^{\alpha}u(x_{i+k-m})
=\frac{1}{h^{\alpha}}\!\!\!\sum_{k=0}^{N_x-i+m}\!\!\!q_{k+p-m}^{\alpha}u(x_{i+k-m}),
\end{split}
\end{equation*}
where $q_{k+p-m}^{\alpha}=0$, when $k+p-m<0$, and $p$ is an integer. And the fourth order approximation is
\begin{equation}\label{2.32}
\begin{split}
&_{x}D_{x_R}^{\alpha}u(x_i)=\,_{4R}\widetilde{A}_{p,q,r,s,\overline{p},\overline{q},\overline{r},\overline{s}}^{\alpha}u(x_i)+\mathcal{O}(h^4)
=\frac{1}{h^{\alpha}}\sum_{k=0}^{N_x-i+m}\varphi_k^{\alpha}u(x_{i+k-m})+\mathcal{O}(h^4),
 \end{split}
\end{equation}
where  $\varphi_k^{\alpha}$ is defined by (\ref{2.26}), and the matrices forms are
\begin{equation}\label{2.33}
  \begin{split}
_R\widetilde{A}_{p}^{\alpha}U=\frac{1}{h^{\alpha}}B_{p}^\alpha U, ~~~~~~~ &B_{p}^\alpha=(A_{p}^\alpha)^T;\\
_{2R}\widetilde{A}_{p,q}^{\alpha}U=\frac{1}{h^{\alpha}}B_{p,q}^\alpha U, \quad
&B_{p,q}^\alpha =w_pB_{p}^\alpha+w_qB_{q}^\alpha;\\
_{3R}\widetilde{A}_{p,q,r,s}^{\alpha}U=\frac{1}{h^{\alpha}}B_{p,q,r,s}^\alpha U, \quad
&B_{p,q,r,s}^\alpha =w_{p,q}B_{p,q}+w_{r,s}B_{r,s};\\
_{4R}\widetilde{A}_{p,q,r,s,\overline{p},\overline{q},\overline{r},\overline{s}}^{\alpha}U
=\frac{1}{h^{\alpha}}B_{p,q,r,s,\overline{p},\overline{q},\overline{r},\overline{s}}^\alpha U, \quad
&B_{p,q,r,s,\overline{p},\overline{q},\overline{r},\overline{s}}^\alpha =w_{p,q,r,s}B_{p,q,r,s}^\alpha
+w_{\overline{p},\overline{q},\overline{r},\overline{s}}B_{\overline{p},\overline{q},\overline{r},\overline{s}}^\alpha.
\end{split}
\end{equation}
\begin{remark}\label{remark2.1}
When $p=0$,  then  $A_p^{\alpha}$ in (\ref{2.28}) reduces to the lower triangular matrix,
and it can be easily checked that all the eigenvalues of $A_p^{\alpha}$ are greater than one;
in fact, from  Lemma \ref{lemma2.1}, it can be noted that $\lambda  (A_p^{\alpha}) =\left(\frac{3}{2}\right)^{\alpha}$, with $\alpha \in(1,2)$.
This is the reason that the scheme for time dependent problem is unstable when directly using the second order Lubich formula with $\alpha \in(1,2)$ to discretize space fractional derivative.
\end{remark}

\subsection{Effective fourth order discretization for space fractional derivatives}
This subsection focuses on how to choose the parameters $p,q,r,s,\overline{p},\overline{q},\overline{r},\overline{s}$ such that all the eigenvalues of the matrix $A_{p,q}^\alpha$ (or $A_{p,q,r,s}^\alpha$ or $A_{p,q,r,s,\overline{p},\overline{q},\overline{r},\overline{s}}^\alpha$) have negative real parts; this means that the corresponding schemes work for space fractional derivatives. Since $B_{p,q}^\alpha$, $B_{p,q,r,s}^\alpha$, and $B_{p,q,r,s,\overline{p},\overline{q},\overline{r},\overline{s}}^\alpha$ is, respectively, the transpose of $A_{p,q}^\alpha$, $A_{p,q,r,s}^\alpha$, and $A_{p,q,r,s,\overline{p},\overline{q},\overline{r},\overline{s}}^\alpha$, we don't need to discuss them separately.

\begin{definition}\cite[p.\,27]{Quarteroni:07}\label{definition2.7}
A matrix $A \in \mathbb{R}^{n\times n}$ is positive definite in $\mathbb{R}^{n}$ if $(Ax,x)>0$, $\forall x \in \mathbb{R}^{n}$, $x\neq 0$.
\end{definition}

\begin{lemma}\cite[p.\,28]{Quarteroni:07}\label{lemma2.4}
A real matrix $A$ of order $n$ is positive definite  if and only if  its symmetric part $H=\frac{A+A^T}{2}$ is positive definite.
Let $H \in \mathbb{R}^{n\times n}$ be symmetric, then $H$ is positive definite if and only if the eigenvalues of $H$ are positive.
\end{lemma}

\begin{lemma}\cite[p.\,184]{Quarteroni:07}\label{lemma2.5}
If $A \in \mathbb{C}^{n \times n}$, let $H=\frac{A+A^H}{2}$ be the hermitian part of $A$,  then for any eigenvalue $\lambda$ of  $A$,
the real part $\Re(\lambda(A))$ satisfies
\begin{equation*}
  \lambda_{min}(H) \leq \Re(\lambda(A)) \leq \lambda_{max}(H),
\end{equation*}
where $\lambda_{min}(H)$ and $\lambda_{max}(H)$ are the minimum and maximum of the eigenvalues of $H$, respectively.
\end{lemma}

\begin{definition}\cite[p.\,13]{Chan:07}\label{definition2.8}
Let  $n \times n$ Toeplitz  matrix  $T_n$ be of the following form:
\begin{equation*}
T_n=\left [ \begin{matrix}
                      t_0           &      t_{-1}             &      \cdots         &       t_{2-n}       &       t_{1-n}      \\
                      t_{1}         &      t_{0}              &      t_{-1}         &      \cdots         &       t_{2-n}        \\
                     \vdots         &      t_{1}              &      t_{0}          &      \ddots         &        \vdots            \\
                     t_{n-2}        &      \cdots             &      \ddots         &      \ddots         &        t_{-1}    \\
                     t_{n-1}        &       t_{n-2}           &      \cdots         &       t_1           &        t_0
 \end{matrix}
 \right ];
\end{equation*}
i.e., $t_{i,j}=t_{i-j}$ and $T_n$ is constant along its diagonals. Assume that the diagonals $\{t_k\}_{k=-n+1}^{n-1}$ are the Fourier coefficients of a function
$f$, i.e.,
\begin{equation*}
  t_k=\frac{1}{2\pi}\int_{-\pi}^{\pi}f(x)e^{-ikx}dx,
\end{equation*}
then the function $f$ is called the generating function of $T_n$.
\end{definition}

\begin{lemma}\cite[p.\,13-15]{Chan:07}\label{lemma2.6} (Grenander-Szeg\"{o} theorem) Let $T_n$ be given by above matrix with a generating function $f$,
where $f$ is a $2\pi$-periodic continuous real-valued functions defined on $[-\pi,\pi]$.
Let $\lambda_{min}(T_n)$ and $\lambda_{max}(T_n)$ denote the smallest and largest eigenvalues of $T_n$, respectively. Then we have
\begin{equation*}
  f_{min} \leq \lambda_{min}(T_n) \leq \lambda_{max}(T_n) \leq f_{max},
\end{equation*}
where $f_{min}$ and  $f_{max}$  is the minimum and maximum values of $f(x)$, respectively.
Moreover, if $f_{min}< f_{max}$, then all eigenvalues of $T_n$ satisfies
\begin{equation*}
  f_{min} < \lambda(T_n) < f_{max},
\end{equation*}
for all $n>0$; In particular, if  $f_{min}>0$, then $T_n$ is positive definite.
\end{lemma}


\begin{theorem}\label{theorem2.9} (Effective second order schemes)
Let $A_{p,q}^\alpha$ be given in (\ref{2.29}) and   $1<\alpha<2$.
Then any eigenvalue $\lambda$ of $A_{p,q}^\alpha$ satisfies
 $$\Re(\lambda(A_{p,q}^\alpha))<0 \quad \mbox{for} \quad (p,q)=(1,q),\quad  |q|\geq 2,$$
moreover, the matrices $A_{p,q}^\alpha$ and $(A_{p,q}^\alpha)^T$ are negative definite.
\end{theorem}
\begin{proof}
\begin{description}
 \item[(1)] For $(p,q)=(1,q)$,  $q\leq -2$, we have   $A_{p,q}^\alpha=\frac{1}{q-1}(qA_1^\alpha-A_q^\alpha)$, and
\end{description}
\begin{equation*}
A_{p,q}^\alpha=\left [ \begin{matrix}
\phi_1^{\alpha}             &\phi_0^{\alpha}         &                          &                          &                             \\
\phi_2^{\alpha}             &\phi_1^{\alpha}         &\phi_0^{\alpha}           &                          &                            \\
\vdots                      &\ddots                  &       \ddots             &        \ddots            &     \\
\phi_{N_x-2}^{\alpha}       &\ddots                  &        \ddots            &   \phi_1^{\alpha}        & \phi_0^{\alpha} \\
\phi_{N_x-1}^{\alpha}       &\phi_{N_x-2}^{\alpha}   &         \cdots           &\phi_2^{\alpha}           & \phi_1^{\alpha}
 \end{matrix}
 \right ],
\end{equation*}
with
 \begin{equation*}
\phi_k^{\alpha}=\left\{ \begin{array}
 {l@{\quad} l}
\frac{qq_k^{\alpha}}{q-1},&0\leq k\leq -q,\\
\frac{qq_k^{\alpha}-q_{k+q-1}^{\alpha}}{q-1},&k>-q.
 \end{array}
 \right.
\end{equation*}
The generating functions of $A_{p,q}^\alpha$ and $(A_{p,q}^{\alpha})^T$ are
\begin{equation*}
  f_{A_{p,q}^\alpha}(x)=\sum_{k=0}^{\infty}\phi_k^{\alpha}e^{i(k-1)x}~~\mbox{and}~~ f_{(A_{p,q}^\alpha)^T}(x)=\sum_{k=0}^{\infty}\phi_k^{\alpha}e^{-i(k-1)x},
\end{equation*}
respectively.  Taking $H_{p,q}=\frac{A_{p,q}^\alpha+\left(A_{p,q}^\alpha\right)^T}{2}$,
 then $f_{p,q}(\alpha,x)=\frac{f_{A_{p,q}^\alpha}(x)+f_{(A_{p,q}^\alpha)^T}(x)}{2}$ is the generating function of $H_{p,q}$.
Since $f_{A_{p,q}^\alpha}(x)$ and $f_{(A_{p,q}^\alpha)^T}(x)$ are mutually conjugated, then  $f_{p,q}(\alpha,x)$ is a $2\pi$-periodic continuous real-valued functions defined on $[-\pi,\pi]$. Moreover, $f_{p,q}(\alpha,x)$ is an even function, so we just need to consider its principal value on $[0,\pi]$.
Next, we prove $f_{p,q}(\alpha,x)\leq 0$. Rephrasing the generating function leads to
\begin{equation*}
\begin{split}
f_{p,q}(\alpha,x)&=\frac{1}{2}\left(\sum_{k=0}^{\infty}\phi_k^{\alpha}e^{i(k-1)x}+\sum_{k=0}^{\infty}\phi_k^{\alpha}e^{-i(k-1)x}\right)\\
   &=\frac{1}{2(q-1)}\left(qe^{-ix}\sum_{k=0}^{\infty}q_k^{\alpha}e^{ikx}-e^{-iqx}\sum_{k=0}^{\infty}q_k^{\alpha}e^{ikx}
   +qe^{ix}\sum_{k=0}^{\infty}q_k^{\alpha}e^{-ikx}-e^{iqx}\sum_{k=0}^{\infty}q_k^{\alpha}e^{-ikx}\right)\\
   &=\frac{1}{2(q-1)} \left[qe^{-ix}(1-e^{ix})^{\alpha} \left(1 +\frac{1}{2}(1-e^{ix})\right)^{\alpha}
     + qe^{ix}(1-e^{-ix})^{\alpha} \left(1 + \frac{1}{2}(1-e^{-ix})\right)^{\alpha} \right]   \\
        &\quad -\frac{1}{2(q-1)} \left[e^{-iqx}(1-e^{ix})^{\alpha} \left(1 + \frac{1}{2}(1-e^{ix})\right)^{\alpha}
       + e^{iqx}(1-e^{-ix})^{\alpha} \left(1 + \frac{1}{2}(1-e^{-ix})\right)^{\alpha} \right].
\end{split}
\end{equation*}
Because of
\begin{equation*}
  (1-e^{\pm ix})^{\alpha}=\left(2sin\frac{x}{2}\right)^{\alpha}e^{\pm i\alpha(\frac{x}{2}-\frac{\pi}{2})},
  ~~\left(1 + \frac{1}{2}(1-e^{\pm ix})\right)^{\alpha}=\left( 1+3sin^2\,\frac{x}{2}\right)^{\frac{\alpha}{2}}e^{\pm i\alpha(\frac{x}{2}-\theta)},
\end{equation*}
where
\begin{equation*}
  \theta=2 arctan \frac{2sin\frac{x}{2}}{cos\frac{x}{2}+ \sqrt{1+3sin^2\,\frac{x}{2}}} \in [0,\pi/2],
\end{equation*}
then, for $q\leq -2$, there exists
\begin{equation*}
\begin{split}
f_{p,q}(\alpha,x)&=\frac{1}{q-1}\left(2sin\frac{x}{2}\right)^{\alpha}\left( 1+3sin^2\,\frac{x}{2}\right)^{\frac{\alpha}{2}}  \left[q\,cos\left (\alpha(x-\frac{\pi}{2}-\theta)-x\right )-cos\left (\alpha(x-\frac{\pi}{2}-\theta)-qx\right)\right].
\end{split}
\end{equation*}

\begin{description}
 \item[(2)] For $(p,q)=(1,q)$,  $q\geq 2$, we have   $A_{p,q}^\alpha=\frac{1}{q-1}(qA_1^\alpha-A_q^\alpha)$ and
\end{description}

\begin{equation*}
A_{p,q}^\alpha=\left [ \begin{matrix}
\phi_q^\alpha      &\phi_{q-1}^\alpha   & \cdots          &\phi_0^\alpha     &               &                  &                 &        \\
\phi_{q+1}^\alpha  &\phi_q^\alpha       &\phi_{q-1}^\alpha&  \cdots          & \phi_0^\alpha &                  &                 &         \\
\phi_{q+2}^\alpha  &\phi_{q+1}^\alpha   &\phi_q^\alpha    &\phi_{q-1}^\alpha &\cdots         &\phi_0^\alpha     &                 &        \\
\vdots             &    \ddots          &\ddots           & \ddots           & \ddots        &   \cdots         & \ddots          &             \\
\phi_{n-2}^\alpha  &\cdots              & \ddots          &\phi_{q+1}^\alpha &\phi_{q}^\alpha&\phi_{q-1}^\alpha & \cdots          &\phi_0^\alpha \\
\vdots             &\ddots              & \cdots          & \ddots           & \ddots        &   \ddots         & \ddots          &\vdots  \\
\phi_{q+n-3}^\alpha&\cdots              &  \ddots         & \cdots           & \ddots        &\phi_{q+1}^\alpha &\phi_{q}^\alpha  &\phi_{q-1}^\alpha \\
\phi_{q+n-2}^\alpha&\phi_{q+n-3}^\alpha &  \cdots         &\phi_{n-2}^\alpha & \cdots        &\phi_{q+2}^\alpha &\phi_{q+1}^\alpha&\phi_{q}^\alpha
 \end{matrix}
 \right ],
\end{equation*}
with
 \begin{equation*}
\phi_k^{\alpha}=\left\{ \begin{array}
 {l@{\quad} l}
-\frac{q_k^{\alpha}}{q-1},&0\leq k\leq q-2,\\
\frac{qq_{k-q+1}^{\alpha}-q_{k}^{\alpha}}{q-1},&k>q-2.
 \end{array}
 \right.
\end{equation*}
The generating functions of $A_{p,q}^\alpha$ and $(A_{p,q}^{\alpha})^T$ are
\begin{equation*}
  f_{A_{p,q}^\alpha}(x)=\sum_{k=0}^{\infty}\phi_k^{\alpha}e^{i(k-q)x}~~\mbox{and}~~ f_{(A_{p,q}^\alpha)^T}(x)=\sum_{k=0}^{\infty}\phi_k^{\alpha}e^{-i(k-q)x},
\end{equation*}
respectively.  Denoting
\begin{equation}\label{2.34}
H_{p,q}=\frac{A_{p,q}^\alpha+\left(A_{p,q}^\alpha\right)^T}{2},
\end{equation}
 then $f_{p,q}(\alpha,x)=\frac{f_{A_{p,q}^\alpha}(x)+f_{(A_{p,q}^\alpha)^T}(x)}{2}$ is the generating function of $H_{p,q}$.
By the similar way, for $q \geq 2$, there exists
\begin{equation*}
\begin{split}
f_{p,q}(\alpha,x)&=\frac{1}{q-1}\left(2sin\frac{x}{2}\right)^{\alpha}\left( 1+3sin^2\,\frac{x}{2}\right)^{\frac{\alpha}{2}}  \left[q\,cos\left (\alpha(x-\frac{\pi}{2}-\theta)-x\right )-cos\left (\alpha(x-\frac{\pi}{2}-\theta)-qx\right)\right].
\end{split}
\end{equation*}
It can be noted that $f_{p,q}(\alpha,x)$ has the same form when  $q \leq -2$ and $q \geq 2$, $p=1$.  And we can check that, for $(p,q)=(1,q)$,  $|q| \geq 2$, there exists  (see Figs. \ref{FIG.1}-\ref{FIG.2})
\begin{equation}\label{2.35}
\begin{split}
f_{p,q}(\alpha,x)&=\frac{1}{q-1}\left(2sin\frac{x}{2}\right)^{\alpha}\left( 1+3sin^2\,\frac{x}{2}\right)^{\frac{\alpha}{2}} \\
 &\quad \cdot \left[q\,cos\left (\alpha(x-\frac{\pi}{2}-\theta)-x\right )-cos\left (\alpha(x-\frac{\pi}{2}-\theta)-qx\right)\right] \leq 0.
\end{split}
\end{equation}
 Since $f_{p,q}(\alpha,x)$ is not identically zero for any given $\alpha \in (1,2)$, from Lemma \ref{lemma2.6},
 it implies that $\lambda (H_{p,q})<0$ and $H_{p,q}$ is negative definite.
Then we get $\Re(\lambda(A_{p,q}^\alpha))<0$ from Lemma  \ref{lemma2.5}, and the matrices $A_{p,q}^\alpha$ and $(A_{p,q}^\alpha)^T$ are negative definite by Lemma \ref{lemma2.4}.
\end{proof}


\begin{figure}[t]
    \begin{minipage}[t]{0.45\linewidth}
    \includegraphics[scale=0.47]{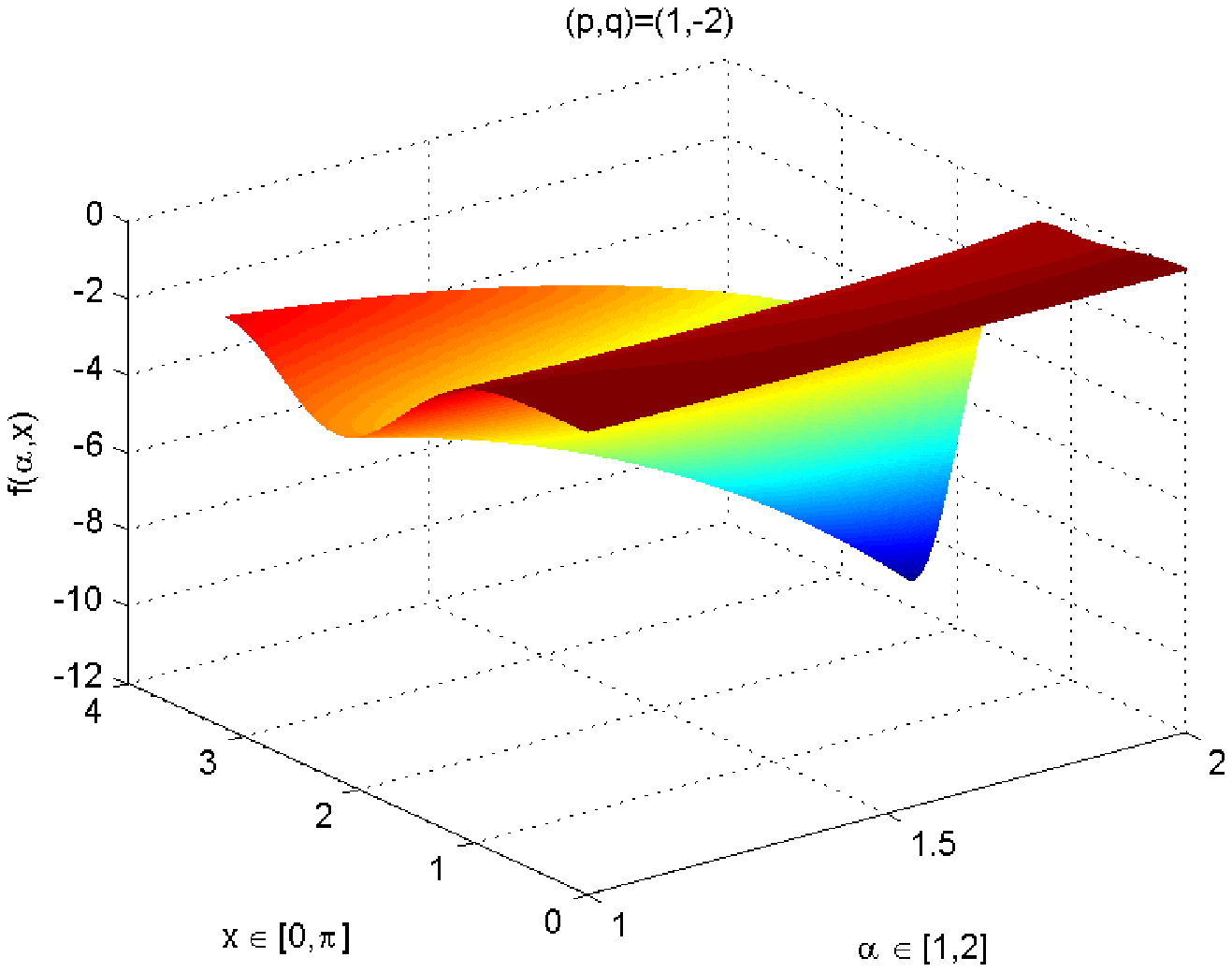}
    \caption{$ f(\alpha,x)$ for $(p,q)=(1,-2)$} \label{FIG.1}
    \end{minipage}
  \begin{minipage}[t]{0.45\linewidth}
    \includegraphics[scale=0.47]{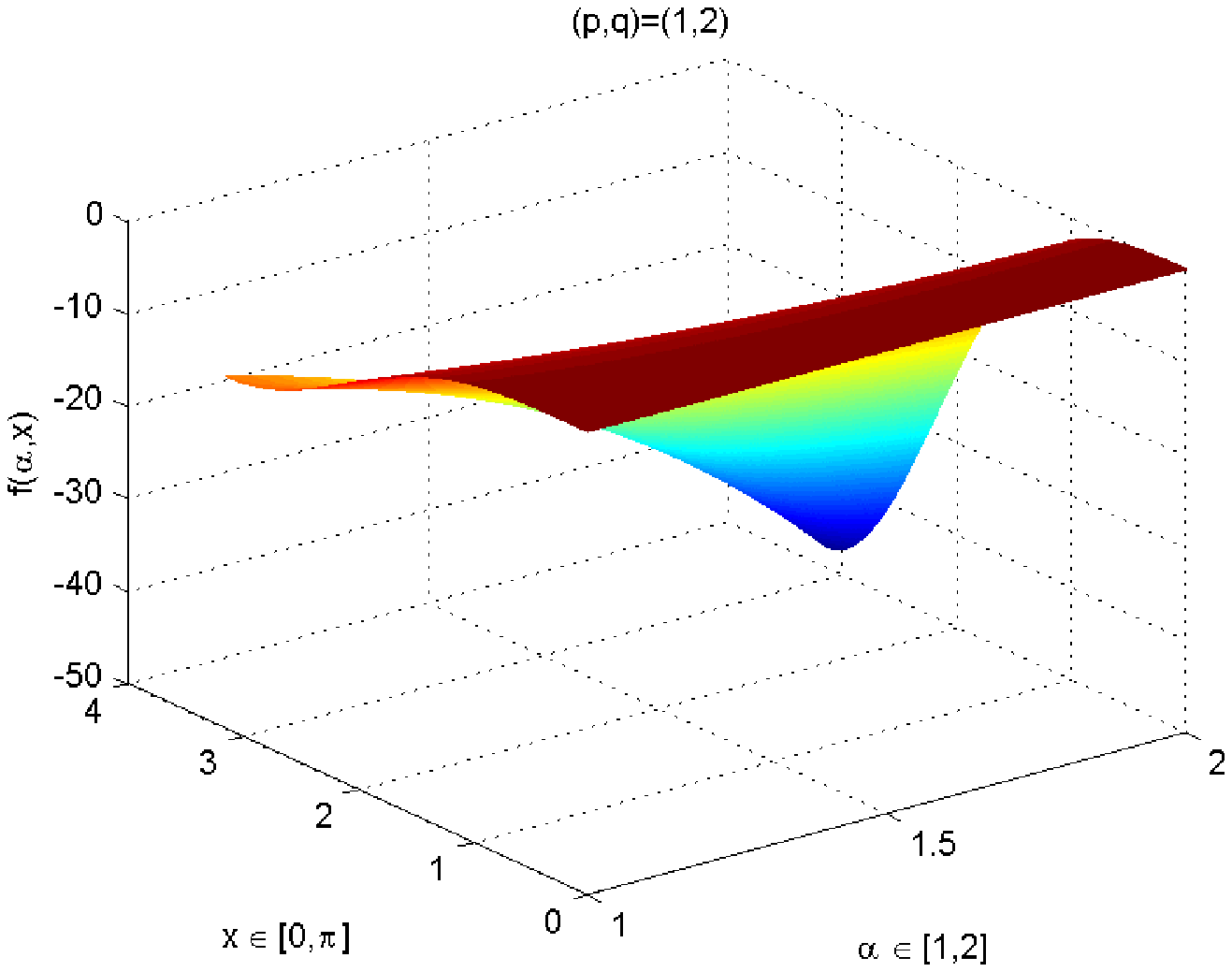}
    \caption{$ f(\alpha,x)\leq 0$ for $(p,q)=(1,2)$} \label{FIG.2}
    \end{minipage}
 \end{figure}

\begin{theorem}\label{theorem2.10} (Effective third order schemes)
Let $A_{p,q,r,s}^\alpha$ with $1<\alpha<2$ be given in (\ref{2.30}).
Then any eigenvalue $\lambda$ of $A_{p,q,r,s}^\alpha$ satisfies
 $$\Re(\lambda(A_{p,q,r,s}^\alpha))<0 \quad \mbox{for} \quad (p,q,r,s)=(1,q,1,s),\quad  |q|\geq 2, \quad  |s|\geq 2,  \quad \mbox{and} \quad  qs<0;$$
moreover, the matrices $A_{p,q,r,s}^\alpha$ and $(A_{p,q,r,s}^\alpha)^T$ are negative definite.
\end{theorem}
\begin{proof}  Taking
\begin{equation}\label{2.36}
 H_{p,q,r,s}=\frac{A_{p,q,r,s}^\alpha+\left(A_{p,q,r,s}^\alpha\right)^T}{2}=w_{p,q}\,H_{p,q}+w_{r,s}\,H_{r,s},
 \end{equation}
 where $H_{p,q}$ and  $H_{r,s}$ are defined by (\ref{2.34}), then
 \begin{equation}\label{2.37}
   f_{p,q,r,s}(\alpha,x)=w_{p,q}\,f_{p,q}(\alpha,x)+w_{r,s}\,f_{r,s}(\alpha,x)
 \end{equation}
 is the generating function of $H_{p,q,r,s} $,
 where $f_{p,q}(\alpha,x)$ and  $f_{r,s}(\alpha,x)$ are given by (\ref{2.35}).
Since $|q|\geq 2$,  $|s|\geq 2$, and  $qs<0$, we can check that
$w_{p,q}=w_{1,q}=\frac{3s+2\alpha}{3(s-q)}>0$,   $w_{r,s}=w_{1,s}=\frac{3q+2\alpha}{3(q-s)}>0$.
 Then from  (\ref{2.35}) and  (\ref{2.37}), we get  $f_{p,q,r,s}(\alpha,x)\leq 0$.

Again, from Lemmas \ref{lemma2.4}-\ref{lemma2.6}, the desired results are obtained.
%
\end{proof}

\begin{theorem}\label{theorem2.11} (Effective fourth order schemes)
Let $A_{p,q,r,s,\overline{p},\overline{q},\overline{r},\overline{s}}^\alpha
$ with $1<\alpha<2$ be given in (\ref{2.31}), where $ (p,q,r,s,\overline{p},\overline{q},\overline{r},\overline{s})
=(1,2,1,-2,1,\overline{q},1,\overline{s})$,
  $|\overline{q}|\geq 2$, $|\overline{s}|\geq 2$, $(\overline{q},\overline{s})\neq (2,-2)$
and $\overline{q}\,\overline{s}<0$.
Then any eigenvalue $\lambda$ of $A_{p,q,r,s,\overline{p},\overline{q},\overline{r},\overline{s}}^\alpha$ satisfies
\begin{equation*}
\Re(\lambda(A_{p,q,r,s,\overline{p},\overline{q},\overline{r},\overline{s}}^\alpha))<0,
 \end{equation*}
and  the matrices $A_{p,q,r,s,\overline{p},\overline{q},\overline{r},\overline{s}}^\alpha$
 and $(A_{p,q,r,s,\overline{p},\overline{q},\overline{r},\overline{s}}^\alpha)^T$ are negative definite.

Moreover, if  $(p,q,r,s,\overline{p},\overline{q},\overline{r},\overline{s})$  takes the following values
 \begin{equation*}
   \begin{split}
& (p,q,r,s,\overline{p},\overline{q},\overline{r},\overline{s})=(1,2,1,0,1,2,1,-2),\\
& (p,q,r,s,\overline{p},\overline{q},\overline{r},\overline{s})=(1,2,1,0,1,-1,1,-2),\\
& (p,q,r,s,\overline{p},\overline{q},\overline{r},\overline{s})=(1,2,1,-1,1,2,1,-2),\\
& (p,q,r,s,\overline{p},\overline{q},\overline{r},\overline{s})=(1,2,1,-1,1,-1,1,-2),\\
& (p,q,r,s,\overline{p},\overline{q},\overline{r},\overline{s})=(1,0,1,-1,1,2,1,-2),\\
& (p,q,r,s,\overline{p},\overline{q},\overline{r},\overline{s})=(1,0,1,-2,1,2,1,-2),\\
& (p,q,r,s,\overline{p},\overline{q},\overline{r},\overline{s})=(1,-1,1,-2,1,2,1,-2),\\
   \end{split}
 \end{equation*}
 then $\Re(\lambda(A_{p,q,r,s,\overline{p},\overline{q},\overline{r},\overline{s}}^\alpha))<0$ and the matrices $A_{p,q,r,s,\overline{p},\overline{q},\overline{r},\overline{s}}^\alpha$
and $(A_{p,q,r,s,\overline{p},\overline{q},\overline{r},\overline{s}}^\alpha)^T$ are negative definite.
\end{theorem}
\begin{proof}
By the similar way to the proofs of Theorems \ref{theorem2.9} and \ref{theorem2.10}, we obtain the desired results.
\end{proof}

\section{Application to the space fractional diffusion equations: the one dimensional case of (\ref{1.2}) and  (\ref{1.2}) itself}
We use two subsections to derive the full discretization of (\ref{1.2}).  First, we present the scheme for the one dimensional case of (\ref{1.2}). The second subsection detailedly provides the full discrete scheme of the two-dimensional  fractional  diffusion equation (\ref{1.2}) with variable coefficients.

\subsection{Numerical scheme  for 1D}
In this subsection, we consider the one-dimensional case of (\ref{1.2}) with variable coefficients, namely,
\begin{equation}\label{3.1}
\frac{\partial u(x,t) }{\partial t}=d_{+}(x) _{x_L}D_x^{\alpha}u(x,t)+d_{-}(x)  _{x}D_{x_R}^{\alpha}u(x,t) +f(x,t).
\end{equation}
In the time direction, we use the Crank-Nicolson scheme.
 The fourth order   left fractional approximation operator
(\ref{2.25}), and right fractional approximation operator
(\ref{2.32}) are respectively used to discretize the left Riemann-Liouville fractional
derivative, and right Riemann-Liouville fractional derivative.

Let the mesh points $x_i=x_L+ih$, $i=-m,\ldots, 0,1,\ldots,N_x-1,N_x,\ldots,N_x+m$, where $m$ is defined by (\ref{2.20})
and $t_n=n\tau$, $0\leq n \leq {N_t}$, where $h=(x_R-x_L)/{N_x}$, $\tau=T/{N_t}$,
i.e., $h$ is the uniform space stepsize and $\tau$ the time steplength.
Taking  $u_{i}^n$ as the approximated value of $u(x_i,t_n)$ and  $d_{+,i}=d_{+}(x_i)$, $d_{-,i}=d_{-}(x_i)$,
$f_i^{n+1/2}=f(x_i,t_{n+1/2})$, where $t_{n+1/2}=(t_n+t_{n+1})/2$.
Then,  Eq. (\ref{3.1}) can be rewritten as
\begin{equation}\label{3.2}
\begin{split}
\frac{u(x_i,t_{n+1})-u(x_i,t_{n})}{\tau}=
&\frac{1}{2}\Big[d_{+,i}\,_{4L}\widetilde{A}_{p,q,r,s,\overline{p},\overline{q},\overline{r},\overline{s}}^{\alpha}u(x_i,t_{n+1})
                +d_{+,i}\,_{4L}\widetilde{A}_{p,q,r,s,\overline{p},\overline{q},\overline{r},\overline{s}}^{\alpha}u(x_i,t_{n})\\
&\quad               +d_{-,i}\,_{4R}\widetilde{A}_{p,q,r,s,\overline{p},\overline{q},\overline{r},\overline{s}}^{\alpha}u(x_i,t_{n+1})
                +d_{-,i}\,_{x_R} \,_{4R}\widetilde{A}_{p,q,r,s,\overline{p},\overline{q},\overline{r},\overline{s}}^{\alpha}u(x_i,t_{n}) \Big]\\
&\quad          +f(x_i,t_{n+1/2})+\mathcal{O}(\tau^2+h^4).
\end{split}
\end{equation}
Multiplying (\ref{3.2}) by $\tau$, we have the following equation
\begin{equation}\label{3.3}
\begin{split}
&\left[ 1- \frac{\tau }{2} \left( d_{+,i}\,_{4L}\widetilde{A}_{p,q,r,s,\overline{p},\overline{q},\overline{r},\overline{s}}^{\alpha}
                                 + d_{-,i}\,_{4R}\widetilde{A}_{p,q,r,s,\overline{p},\overline{q},\overline{r},\overline{s}}^{\alpha} \right)\right] u(x_i,t_{n+1})\\
&\quad=\left[ 1+ \frac{\tau }{2}\left( d_{+,i}\,_{4L}\widetilde{A}_{p,q,r,s,\overline{p},\overline{q},\overline{r},\overline{s}}^{\alpha}
                                + d_{-,i}\,_{4R}\widetilde{A}_{p,q,r,s,\overline{p},\overline{q},\overline{r},\overline{s}}^{\alpha} \right)\right]u(x_i,t_{n})
                                +\tau f(x_i,t_{n+1/2})+ R_i^{n+1},
\end{split}
\end{equation}
with
 \begin{equation}\label{3.4}
 |R_i^{n+1}|\leq \widetilde{c} \tau\left(\tau^2+h^4\right).
 \end{equation}
Therefore, the full discretization of
 (\ref{3.3}) has the following form
 \begin{equation}\label{3.5}
\begin{split}
&\left[ 1- \frac{\tau }{2} \left( d_{+,i}\,_{4L}\widetilde{A}_{p,q,r,s,\overline{p},\overline{q},\overline{r},\overline{s}}^{\alpha}
                                 + d_{-,i}\,_{4R}\widetilde{A}_{p,q,r,s,\overline{p},\overline{q},\overline{r},\overline{s}}^{\alpha} \right)\right]  u_i^{n+1}\\
&\quad=\left[ 1+ \frac{\tau }{2}\left( d_{+,i}\,_{4L}\widetilde{A}_{p,q,r,s,\overline{p},\overline{q},\overline{r},\overline{s}}^{\alpha}
                                + d_{-,i}\,_{4R}\widetilde{A}_{p,q,r,s,\overline{p},\overline{q},\overline{r},\overline{s}}^{\alpha} \right)\right] u_i^{n}
                                +\tau  f_i^{n+1/2},
\end{split}
\end{equation}
and it can be rewritten as
 \begin{equation}\label{3.6}
\begin{split}
 &u_i^{n+1} - \frac{\tau }{2} \left[\frac{ d_{+,i}}{h^{\alpha}}\sum_{k=0}^{i+m} \varphi_k^{\alpha}u_{i-k+m}^{n+1}
   +\frac{d_{-,i} }{h^{\alpha}} \sum_{k=0}^{N_x-i+m}\varphi_k^{\alpha}u_{i+k-m}^{n+1}    \right]     \\
 &\quad = u_i^{n} + \frac{\tau }{2}\left[ \frac{ d_{+,i}}{h^{\alpha}}\sum_{k=0}^{i+m} \varphi_k^{\alpha}u_{i-k+m}^{n}
          + \frac{d_{-,i}}{h^{\alpha}} \sum_{k=0}^{N_x-i+m}\varphi_k^{\alpha}u_{i+k-m}^{n} \right]  +\tau f_i^{n+1/2}.
\end{split}
\end{equation}
For the convenience of implementation, we use the matrix form of the grid functions
 \begin{equation*}
 U^{n}=[u_1^n,u_2^n,\ldots,u_{N_x-1}^n]^{\rm T}, ~~F^{n+1/2}=[f_1^{n+1/2},f_2^{n+1/2},\ldots,f_{N_x-1}^{n+1/2}]^{\rm T},
  \end{equation*}
therefore, the finite difference scheme (\ref{3.6}) can be recast as
 \begin{equation}\label{3.7}
\begin{split}
 &  \left[I - \frac{\tau }{2h^{\alpha}}\left(  D_{+}A_{\alpha}+D_{-}A_{\alpha}^T \right) \right] U^{n+1}
= \left[I+  \frac{\tau }{2h^{\alpha}}\left(  D_{+}A_{\alpha}+D_{-}A_{\alpha}^T  \right) \right] U^{n}+\tau F^{n+1/2},
\end{split}
\end{equation}
where $A_{\alpha}=A_{p,q,r,s,\overline{p},\overline{q},\overline{r},\overline{s}}^\alpha$ is defined by (\ref{2.31}), and
\begin{equation}\label{3.8}
D_{+}=\left [ \begin{matrix}
  d_{+,1}                        \\
       & d_{+,2}                  \\
       &       & \ddots            \\
       &       &      &d_{+,N_x-1}
 \end{matrix}
 \right ],
 \quad
D_{-}=\left [ \begin{matrix}
  d_{-,1}                        \\
       & d_{-,2}                  \\
       &       & \ddots            \\
       &       &      &d_{-,N_x-1}
 \end{matrix}
 \right ].
\end{equation}

\subsection{Numerical scheme  for 2D}

We now examine the full discretization scheme of (\ref{1.2}). For effectively performing the theoretical analysis,
we suppose  $d_{+}(x)=d_{+}(x,y)$,  $d_{-}(x)=d_{-}(x,y)$, and  $e_{+}(y)=e_{+}(x,y)$, $e_{-}(y)=e_{-}(x,y)$.

Analogously we still use the Crank-Nicolson scheme to do the discretization in time direction.
Let the mesh points $x_i=x_L+ih$, $i=-m,\ldots, 0,1,\ldots,N_x-1,N_x,\ldots,N_x+m$,
 and $y_j=y_L+j\Delta y$, $j=-m,\ldots, 0,1,\ldots,N_y-1,N_y,\ldots,N_y+m$, where $m$ is given in (\ref{2.20}), $t_n=n\tau$, $0\leq n \leq {N_t}$, and
 $\Delta x=(x_R-x_L)/{N_x}$, $\Delta y=(y_R-y_L)/{N_y}$, $\tau=T/{N_t}$;
and  $d_{+,i}=d_{+}(x_i,y_j)$, $d_{-,i}=d_{-}(x_i,y_j)$, and $e_{+,j}=e_{+}(x_i,y_j)$, $e_{-,j}=e_{-}(x_i,y_j)$.
Taking $u_{i,j}^n$ as the approximated value of $u(x_i,y_j,t_n)$ and
$f_{i,j}^{n+1/2}=f(x_i,y_j,t_{n+1/2})$, where $t_{n+1/2}=(t_n+t_{n+1})/2$.
Then, Eq. (\ref{1.2}) can be rewritten as

 \begin{equation}\label{3.9}
\begin{split}
&\left[\! 1\!-\! \frac{\tau }{2} \left(\! d_{+,i}\,_{4L}\widetilde{A}_{p,q,r,s,\overline{p},\overline{q},\overline{r},\overline{s}}^{\alpha}
                                \! +\! d_{-,i}\,_{4R}\widetilde{A}_{p,q,r,s,\overline{p},\overline{q},\overline{r},\overline{s}}^{\alpha}
                                \! + \! e_{+,j}\,_{4L}\widetilde{A}_{p,q,r,s,\overline{p},\overline{q},\overline{r},\overline{s}}^{\beta}
       \! +\! e_{-,j}\,_{4R}\widetilde{A}_{p,q,r,s,\overline{p},\overline{q},\overline{r},\overline{s}}^{\beta} \!\right)\!\right]\!u(x_i,y_j,t_{n+1})\\
&=\!\!\left[\! 1\!+\! \frac{\tau }{2} \left(\! d_{+,i}\,_{4L}\widetilde{A}_{p,q,r,s,\overline{p},\overline{q},\overline{r},\overline{s}}^{\alpha}
                                \! +\! d_{-,i}\,_{4R}\widetilde{A}_{p,q,r,s,\overline{p},\overline{q},\overline{r},\overline{s}}^{\alpha}
                                \! + \!  e_{+,j}\,_{4L}\widetilde{A}_{p,q,r,s,\overline{p},\overline{q},\overline{r},\overline{s}}^{\beta}
                                \! +\! e_{-,j}\,_{4R}\widetilde{A}_{p,q,r,s,\overline{p},\overline{q},\overline{r},\overline{s}}^{\beta} \!\right)\!\right]\! u(x_i,y_j,t_{n}) \\
 & \qquad                       +\tau f(x_i,y_j,t_{n+1/2})+ R_{i,j}^{n+1},
\end{split}
\end{equation}
with
 \begin{equation}\label{3.10}
 |R_{i,j}^{n+1}|\leq \widetilde{c} \tau\left(\tau^2+(\Delta x)^4+(\Delta y)^4\right).
 \end{equation}
Then, the resulting discretization of (\ref{3.9}) has the following form

 \begin{equation}\label{3.11}
\begin{split}
&\left[ 1- \frac{\tau }{2} \left( d_{+,i}\,_{4L}\widetilde{A}_{p,q,r,s,\overline{p},\overline{q},\overline{r},\overline{s}}^{\alpha}
                                 + d_{-,i}\,_{4R}\widetilde{A}_{p,q,r,s,\overline{p},\overline{q},\overline{r},\overline{s}}^{\alpha} +
e_{+,j}\,_{4L}\widetilde{A}_{p,q,r,s,\overline{p},\overline{q},\overline{r},\overline{s}}^{\beta}
        + e_{-,j}\,_{4R}\widetilde{A}_{p,q,r,s,\overline{p},\overline{q},\overline{r},\overline{s}}^{\beta} \right)\right] u_{i,j}^{n+1}\\
&\quad=\left[ 1+\frac{\tau }{2} \left( d_{+,i}\,_{4L}\widetilde{A}_{p,q,r,s,\overline{p},\overline{q},\overline{r},\overline{s}}^{\alpha}
                                 + d_{-,i}\,_{4R}\widetilde{A}_{p,q,r,s,\overline{p},\overline{q},\overline{r},\overline{s}}^{\alpha}
         +  e_{+,j}\,_{4L}\widetilde{A}_{p,q,r,s,\overline{p},\overline{q},\overline{r},\overline{s}}^{\beta}
                                 + e_{-,j}\,_{4R}\widetilde{A}_{p,q,r,s,\overline{p},\overline{q},\overline{r},\overline{s}}^{\beta} \right)\right]u_{i,j}^{n} \\
 & \qquad                       +\tau f_{i,j}^{n+1/2}.
\end{split}
\end{equation}
We further define
\begin{equation*}
\begin{split}
 &\delta_{\alpha,x}:= d_{+,i}\,_{4L}\widetilde{A}_{p,q,r,s,\overline{p},\overline{q},\overline{r},\overline{s}}^{\alpha}
                                 + d_{-,i}\,_{4R}\widetilde{A}_{p,q,r,s,\overline{p},\overline{q},\overline{r},\overline{s}}^{\alpha};\\
 &\delta_{\beta,y}:=e_{+,j}\,_{4L}\widetilde{A}_{p,q,r,s,\overline{p},\overline{q},\overline{r},\overline{s}}^{\beta}
        + e_{-,j}\,_{4R}\widetilde{A}_{p,q,r,s,\overline{p},\overline{q},\overline{r},\overline{s}}^{\beta},
\end{split}
\end{equation*}
thus Eq. (\ref{3.11}) can be rewritten as
\begin{equation}\label{3.12}
 \left (1-\frac{\tau}{2}\delta_{\alpha,x}-\frac{\tau}{2}\delta_{\beta,y}\right )u_{i,j}^{n+1}
 =\left (1+\frac{\tau}{2}\delta_{\alpha,x}+\frac{\tau}{2}\delta_{\beta,y}\right )u_{i,j}^{n}+ \tau f_{i,j}^{n+1/2}.
\end{equation}
The perturbation equation of (\ref{3.12}) is of the form
\begin{equation}\label{3.13}
\begin{split}
 \left (1-\frac{\tau}{2}\delta_{\alpha,x} \right )\left(1-\frac{\tau}{2}\delta_{\beta,y}\right )u_{i,j}^{n+1}
 =\left (1+\frac{\tau}{2}\delta_{\alpha,x}\right)\left(1+\frac{\tau}{2}\delta_{\beta,y}\right )u_{i,j}^{n}+ \tau f_{i,j}^{n+1/2}.
\end{split}
\end{equation}
Comparing (\ref{3.13}) with (\ref{3.12}), the splitting term is
given by
\begin{equation*}
\frac{\tau^2}{4}\delta_{\alpha,x}\delta_{\beta,y}(u_{i,j}^{n+1}-u_{i,j}^{n}),
\end{equation*}
since $(u_{i,j}^{n+1}-u_{i,j}^n)$ is an
$\mathcal{O}(\tau)$ term, it implies  that this perturbation
contributes an $\mathcal{O}(\tau^2)$ error component.

The system of equations defined by (\ref{3.13}) can be solved by the following schemes.

PR-ADI scheme [\cite{Peaceman:55}]:
\begin{equation}\label{3.14}
\left (1-\frac{\tau}{2}\delta_{\alpha,x} \right )u_{i,j}^{*}
 =\left(1+\frac{\tau}{2}\delta_{\beta,y}\right )u_{i,j}^{n}+ \frac{\tau}{2}
 f_{i,j}^{n+1/2};
\end{equation}
\begin{equation}\label{3.15}
\left(1-\frac{\tau}{2}\delta_{\beta,y}\right )u_{i,j}^{n+1}
=\left (1+\frac{\tau}{2}\delta_{\alpha,x}\right)u_{i,j}^{*}+
\frac{\tau}{2} f_{i,j}^{n+1/2}.
\end{equation}

 D-ADI scheme [\cite{Douglas:55}]:
\begin{equation}\label{3.16}
\left (1-\frac{\tau}{2}\delta_{\alpha,x} \right )u_{i,j}^{*}
 =\left(1+\frac{\tau}{2}\delta_{\alpha,x}+\tau\delta_{\beta,y}\right )u_{i,j}^{n}+\tau f_{i,j}^{n+1/2};
\end{equation}
\begin{equation}\label{3.17}
\left(1-\frac{\tau}{2}\delta_{\beta,y}\right )u_{i,j}^{n+1}
=u_{i,j}^{*}- \frac{\tau}{2}\delta_{\beta,y}u_{i,j}^{n}.
\end{equation}
Take
\begin{equation*}
\begin{split}
&\mathbf{U}^{n}=[u_{1,1}^n,u_{2,1}^n,\ldots,u_{N_x-1,1}^n,u_{1,2}^n,u_{2,2}^n,\dots,u_{N_x-1,2}^n,\ldots,u_{1,N_y-1}^n,u_{2,N_y-1}^n,\ldots,u_{N_x-1,N_y-1}^n]^T,\\
&\mathbf{F}^{n}=[f_{1,1}^n,f_{2,1}^n,\ldots,f_{N_x-1,1}^n,f_{1,2}^n,f_{2,2}^n,\dots,f_{N_x-1,2}^n,\ldots,f_{1,N_y-1}^n,f_{2,N_y-1}^n,\ldots,f_{N_x-1,N_y-1}^n]^T,\\
\end{split}
\end{equation*}
and denote
\begin{equation}\label{3.18}
\begin{split}
&\mathcal{A}_x =\frac{ \tau}{2(\Delta x)^{\alpha}} \left[ (I \otimes D_{+}) (I \otimes A_{\alpha})
               +(I \otimes D_{-}) (I \otimes A_{\alpha}^T) \right]
               =\frac{ \tau}{2(\Delta x)^{\alpha}} I \otimes \left(D_{+} A_{\alpha} + D_{-} A_{\alpha}^T\right),\\
&\mathcal{A}_y =\frac{ \tau}{2(\Delta y)^{\beta}} \left[(E_{+}  \otimes I)  (A_{\beta}  \otimes I)
               +(E_{-}  \otimes I)  ( A_{\beta}^T \otimes I )\right]
               =\frac{ \tau}{2(\Delta y)^{\beta}} \left(E_{+}A_{\beta}+E_{-}A_{\beta}^T\right)  \otimes I,
\end{split}
 \end{equation}
where $I$  denotes the unit matrix and  the symbol $\otimes$ the Kronecker product [see, \cite{Laub:05}],
and $A_{\alpha}=A_{p,q,r,s,\overline{p},\overline{q},\overline{r},\overline{s}}^\alpha$,    $A_{\beta}=A_{p,q,r,s,\overline{p},\overline{q},\overline{r},\overline{s}}^\beta$ are defined by (\ref{2.31}).
The matrices  $D_{+}$  and $D_{-}$ are defined by (\ref{3.8}), and
 \begin{equation}\label{3.19}
E_{+}=\left [ \begin{matrix}
  e_{+,1}                        \\
       & e_{+,2}                  \\
       &       & \ddots            \\
       &       &      &e_{+,N_y-1}
 \end{matrix}
 \right ],
 \quad
E_{-}=\left [ \begin{matrix}
  e_{-,1}                        \\
       & e_{-,2}                  \\
       &       & \ddots            \\
       &       &      &e_{-,N_y-1}
 \end{matrix}
 \right ].
\end{equation}
Therefore, the finite difference scheme (\ref{3.13}) has the following form
\begin{equation}\label{3.20}
  (I-\mathcal{A}_x )(I-\mathcal{A}_y )\mathbf{U}^{n+1}= (I+\mathcal{A}_x )(I+\mathcal{A}_y )\mathbf{U}^{n}+\tau \mathbf{F}^{n+1/2}.
\end{equation}
\begin{remark}\label{remark3.1}
The schemes (\ref{3.14})-(\ref{3.15}) and (\ref{3.16})-(\ref{3.17})
are equivalent, since both of them come from (\ref{3.13}), see [\cite{Deng:11}].
\end{remark}
\section{Convergence and Stability Analysis}
In this section, we  theoretically prove  that the  difference scheme is unconditionally stable  and 4th order convergent
in  space  directions and 2nd order convergent in time direction. In the following, the matrices $D_{+}$, $D_{-}$ and $E_{+}$, $E_{-}$ are defined by
(\ref{3.8}) and (\ref{3.19}), respectively.

\begin{lemma}\cite[p.\,140]{Laub:05}\label{lemma4.1}
Let $A \in \mathbb{R}^{m\times n}$, $B \in \mathbb{R}^{r\times s}$, $C \in \mathbb{R}^{n\times p}$, and $D \in \mathbb{R}^{s\times t}$.
Then
\begin{equation*}
  (A \otimes B)(C \otimes D)=AC \otimes  BD \quad (\in \mathbb{R}^{mr\times pt}).
\end{equation*}
Moreover, for all $A$ and $B$, $(A \otimes B)^T=A^T\otimes B^T$.
\end{lemma}

\begin{lemma}\cite[p.\,141]{Laub:05}\label{lemma4.2}
Let $A \in \mathbb{R}^{n\times n}$ have eigenvalues $\{\lambda_i\}_{i=1}^n$ and $B \in \mathbb{R}^{m\times m}$ have eigenvalues $\{\mu_j\}_{j=1}^m$.
Then the $mn$ eigenvalues of $A \otimes B$ are
\begin{equation*}
  \lambda_1\mu_1,\ldots,\lambda_1\mu_m, \lambda_2\mu_1,\ldots,\lambda_2\mu_m,\ldots,\lambda_n\mu_1\ldots,\lambda_n\mu_m.
\end{equation*}
\end{lemma}

\begin{theorem}\label{theorem4.1}
 Let the  matrix $A_{\alpha}=A_{p,q,r,s,\overline{p},\overline{q},\overline{r},\overline{s}}^\alpha$ be defined by (\ref{2.31})
and  $D_{-}=\kappa_\alpha D_{+}$,  where $\kappa_\alpha$ is any given nonnegative constant.
 Then we have $\Re\left(\lambda\left(D_{+}( A_{\alpha}+ \kappa_\alpha A_{\alpha}^T)\right)\right)<0$.
\end{theorem}
\begin{proof}
  Since
\begin{equation*}
  D_{+}^{-\frac{1}{2}}\left[D_{+}( A_{\alpha}+ \kappa_\alpha A_{\alpha}^T)\right] D_{+}^{\frac{1}{2}}
   =    D_{+}^{\frac{1}{2}}  ( A_{\alpha}+ \kappa_\alpha A_{\alpha}^T) D_{+}^{\frac{1}{2}},
\end{equation*}
it means that $D_{+}( A_{\alpha}+ \kappa_\alpha A_{\alpha}^T)$ and $D_{+}^{\frac{1}{2}} ( A_{\alpha}+ \kappa_\alpha A_{\alpha}^T)D_{+}^{\frac{1}{2}}$  are similar.
From Theorem \ref{theorem2.11}, we know $A_{\alpha}$ and $A_{\alpha}^T$ are negative definite,  and thanks to Definition \ref{definition2.7}, it implies that
\begin{equation*}
  \left ( D_{+}^{\frac{1}{2}}( A_{\alpha}+ \kappa_\alpha A_{\alpha}^T) D_{+}^{\frac{1}{2}}x,x \right)
  =\left( ( A_{\alpha}+ \kappa_\alpha A_{\alpha}^T) D_{+}^{\frac{1}{2}}x,D_{+}^{\frac{1}{2}}x \right)<0, ~~ \forall x \in \mathbb{R}^{n}, x\neq 0,
\end{equation*}
i.e., the  matrix $\widetilde{A}:=D_{+}^{\frac{1}{2}}  ( A_{\alpha}+ \kappa_\alpha A_{\alpha}^T) D_{+}^{\frac{1}{2}}$ is negative definite.
From Lemma \ref{lemma2.4}, $\widetilde{H}=\frac{\widetilde{A}+\widetilde{A}^T}{2}$ is negative definite and $ \lambda_{max}(\widetilde{H})<0$;
and  according to  Lemma \ref{lemma2.5}, we obtain $\Re(\lambda(\widetilde{A})) \leq \lambda_{max}(\widetilde{H})<0$.
Therefore, $\Re\left(\lambda\left(D_{+}( A_{\alpha}+ \kappa_\alpha A_{\alpha}^T)\right)\right)=\Re(\lambda(\widetilde{A}))<0$.
\end{proof}

\begin{theorem}\label{theorem4.2}
Let $\mathcal{A}_x$ and $\mathcal{A}_y$  be defined by (\ref{3.18}) and  $D_{-}=\kappa_\alpha D_{+}$, $E_{-}=\kappa_\beta E_{+}$,
 where $\kappa_\alpha$ and $\kappa_\beta$ are any given nonnegative constants.  Then we have
$\Re\left(\lambda(\mathcal{A}_x )\right)<0$ and $\Re\left(\lambda(\mathcal{A}_y )\right)<0$.
\end{theorem}

\begin{proof}
From (\ref{3.18}), there exists
  \begin{equation*}
\begin{split}
&\mathcal{A}_x =\frac{ \tau}{2(\Delta x)^{\alpha}} I \otimes \left(D_{+} A_{\alpha} + D_{-} A_{\alpha}^T\right)
               =\frac{ \tau}{2(\Delta x)^{\alpha}} I \otimes \left(D_{+}( A_{\alpha}+ \kappa_\alpha A_{\alpha}^T)\right),\\
&\mathcal{A}_y =\frac{ \tau}{2(\Delta y)^{\beta}} \left(E_{+}A_{\beta}+E_{-}A_{\beta}^T\right)  \otimes I
               =\frac{ \tau}{2(\Delta y)^{\beta}} \left(E_{+}( A_{\beta}+ \kappa_\beta A_{\beta}^T)  \right) \otimes I.
\end{split}
 \end{equation*}
By Theorem \ref{theorem4.1}, we get  $\Re\left(\lambda\left(D_{+}( A_{\alpha}+ \kappa_\alpha A_{\alpha}^T)\right)\right)<0$
and $\Re\left(\lambda\left(E_{+}( A_{\beta}+ \kappa_\beta A_{\beta}^T)\right)\right)<0$.
Then,  according to  Lemma \ref{lemma4.2}, it implies that $\Re\left(\lambda(\mathcal{A}_x )\right)<0$ and $\Re\left(\lambda(\mathcal{A}_y )\right)<0$.
\end{proof}

\begin{remark}\label{remark4.1}
If taking $\kappa_\alpha=\kappa_\beta=0$, then Eq. (\ref{1.2}) becomes the one-sided fractional diffusion equation; and
 $\kappa_\alpha=\kappa_\beta=1$, Eq. (\ref{1.2}) reduces to the space-Riesz fractional diffusion equation.
\end{remark}

\subsection{Stability and Convergence for 1D }

\begin{theorem}\label{theorem4.3}
Let $D_{-}=\kappa_\alpha D_{+}$,
then the difference scheme (\ref{3.7})  with $\alpha \in (1,2)$
is unconditionally stable.
\end{theorem}
\begin{proof}
Let $\widetilde{u}_{i}^n~(i=1,2,\ldots,N_x-1;\,n=0,1,\ldots,N_t)$ be the approximate solution of $u_i^n$,
which is the exact solution of the difference scheme (\ref{3.7}).
Putting $\epsilon_i^n=\widetilde{u}_i^n-u_i^n$, and denoting $\epsilon^n=[\epsilon_1^n,\epsilon_2^n,\ldots, \epsilon_{N_x-1}^n]$,
 then from (\ref{3.7}) we obtain the following perturbation equation
 \begin{equation*}
  (I-A)\epsilon^{n+1}=(I+A)\epsilon^n,
\end{equation*}
i.e.,
 \begin{equation*}
 \epsilon^{n+1}= (I-A)^{-1}(I+A)\epsilon^n,
\end{equation*}
with
\begin{equation}\label{4.1}
  A= \frac{\tau }{2h^{\alpha}}D_{+}( A_{\alpha}+ \kappa_\alpha A_{\alpha}^T).
\end{equation}
Denoting $\lambda$ as an eigenvalue of the matrix $A$, then from Theorem \ref{theorem4.1}, we get $\Re(\lambda(A))<0$.
Note that
 $\lambda$ is an eigenvalue of the matrix $A$ if and only if
$1-\lambda$ is an eigenvalue of the matrix $I-A$, if and only if $(1-\lambda)^{-1}(1+\lambda)$ is an
eigenvalue of the matrix $(I-A)^{-1}(I+A)$. Since $\Re(\lambda(A))<0 $, it implies that $|(1-\lambda)^{-1}(1+\lambda)|<1$.
Thus, the spectral radius of the matrix $(I-A)^{-1}(I+A)$ is less than $1$, hence the scheme (\ref{3.7})
is unconditionally stable.

\end{proof}

\begin{theorem}\label{theorem4.4}
Let $u(x_i,t_n)$ be the exact solution of (\ref{3.1})  with $\alpha \in (1,2)$, $u_i^n$ the  solution of
the  finite difference scheme (\ref{3.7}), and $D_{-}=\kappa_\alpha D_{+}$,  then there is a positive constant $C$ such that
\begin{equation*}
  \begin{split}
||u(x_i,t_n)-u_i^n||_2 \leq  C (\tau^2+h^4), \quad i=1,2,\ldots,N_x-1;\,n=0,1,\ldots,N_t.
  \end{split}
  \end{equation*}
\end{theorem}
\begin{proof}
Denoting $e_i^n=u(x_i,t_n)-u_i^n$, and    $e^n=[e_1^n,e_2^n,\ldots, e_{N_x-1}^n]^T$.
 Subtracting (\ref{3.2}) from (\ref{3.7}) and using $e^0=0$, we obtain
 \begin{equation*}
  (I-A)e^{n+1}=(I+A)e^n+R^{n+1},
\end{equation*}
where $A$ is defined by (\ref{4.1}), and $R^n=[R_1^n,R_2^n,\ldots, R_{N_x-1}^n]^T$. The above equation can be rewritten as
  \begin{equation*}
  e^{n+1}=(I-A)^{-1}(I+A)e^n+(I-A)^{-1}R^{n+1}.
\end{equation*}
Similar to the proof of Theorem 4.2 of [\cite{Deng:11}], we have that $\|(I-A)^{-1}(I+A)\|_2$ and $\|(I-A)^{-1}\|_2$ are less than $1$.
 Then, using  $|R_i^{n+1}|\leq \widetilde{c} \tau(\tau^2+h^4)$ in (\ref{3.4}), we obtain
\begin{equation*}
\begin{split}
  ||e^{n}||_2 &\leq ||(I-A)^{-1}(I+A)||_2\cdot||e^{n-1}||_2+||(I-A)^{-1}||_2\cdot|R^{n}|\\
            & \leq  ||e^{n-1}||_2+|R^{n}|\leq\sum_{k=0}^{n-1}|R^{k+1}|\leq  c (\tau^2+h^4).
\end{split}
\end{equation*}
\end{proof}

\subsection{Stability and Convergence for 2D }

\begin{theorem}\label{theorem4.5}
Let $D_{-}=\kappa_\alpha D_{+}$ and $E_{-}=\kappa_\beta E_{+}$,
then the difference scheme (\ref{3.20}) with $1< \alpha,\beta <2$
is unconditionally stable.
\end{theorem}
\begin{proof}
Let $\widetilde{u}_{i,j}^n~(i=1,2,\ldots,N_x-1;j=1,2,\ldots,N_y-1;n=0,1,\ldots,N_t)$ be the approximate solution of ${u}_{i,j}^n$,
which is the exact solution of the  difference scheme (\ref{3.20}).
Taking $\epsilon_{i,j}^n=\widetilde{u}_{i,j}^n-{u}_{i,j}^n$,
 then from (\ref{3.20}) we obtain the following perturbation equation
 \begin{equation}\label{4.2}
  (I-\mathcal{A}_x )(I-\mathcal{A}_y )\mathbf{\bm \epsilon}^{n+1}= (I+\mathcal{A}_x )(I+\mathcal{A}_y )\mathbf{\bm \epsilon}^{n},
\end{equation}
where $\mathcal{A}_x$ and  $\mathcal{A}_y$ are given by (\ref{3.18}), and
\begin{equation*}
\begin{split}
&\mathbf{\bm \epsilon}^{n}=[\epsilon_{1,1}^n,\epsilon_{2,1}^n,\ldots,\epsilon_{N_x-1,1}^n,
   \epsilon_{1,2}^n,\epsilon_{2,2}^n,\dots,\epsilon_{N_x-1,2}^n,\ldots,\epsilon_{1,N_y-1}^n,\epsilon_{2,N_y-1}^n,\ldots,\epsilon_{N_x-1,N_y-1}^n]^T.
\end{split}
\end{equation*}
Then Eq. (\ref{4.2}) can be rewritten as
 \begin{equation}\label{4.3}
  \mathbf{\bm \epsilon}^{n+1}= (I-\mathcal{A}_y )^{-1}(I-\mathcal{A}_x )^{-1}(I+\mathcal{A}_x )(I+\mathcal{A}_y )\mathbf{\bm \epsilon}^{n}.
\end{equation}
According to Lemma \ref{lemma4.1} and (\ref{3.18}), it is easy to check that $\mathcal{A}_x$ and $\mathcal{A}_y$ commute, i.e.,
\begin{equation}\label{4.4}
\begin{split}
\mathcal{A}_x\mathcal{A}_y=\mathcal{A}_y\mathcal{A}_x  =\frac{ \tau^2}{4(\Delta x)^{\alpha}(\Delta y)^{\beta}}
\left(E_{+}A_{\beta}+E_{-}A_{\beta}^T\right)  \otimes \left(D_{+} A_{\alpha} + D_{-} A_{\alpha}^T\right).
\end{split}
 \end{equation}
 Then Eq. (\ref{4.3}) has the following form
  \begin{equation}\label{4.5}
  \mathbf{\bm \epsilon}^{n+1}= (I-\mathcal{A}_x )^{-1}(I+\mathcal{A}_x )(I-\mathcal{A}_y )^{-1}(I+\mathcal{A}_y )\mathbf{\bm \epsilon}^{n}.
\end{equation}
Form Theorem \ref{theorem4.2}, we have
$\Re\left(\lambda(\mathcal{A}_x )\right)<0$ and $\Re\left(\lambda(\mathcal{A}_y )\right)<0$.
Similar to the proof of the Theorem \ref{theorem4.3},
 the spectral radius of the matrix $(I-\mathcal{A}_x)^{-1}(I+\mathcal{A}_x)$ and $(I-\mathcal{A}_y)^{-1}(I+\mathcal{A}_y)$ are less than $1$.
Then the difference scheme (\ref{3.20}) is unconditionally stable.

\end{proof}

\begin{theorem}\label{theorem4.6}
Let $u(x_i,y_j,t_n)$ be the exact solution of (\ref{1.2}) with $1< \alpha,\beta <2$,
 $u_{i,j}^n$ the  solution of
the  finite difference scheme (\ref{3.20}), and $D_{-}=\kappa_\alpha D_{+}$ and $E_{-}=\kappa_\beta E_{+}$,  then there is a positive constant $C$ such that
\begin{equation*}
  \begin{split}
||u(x_i,y_j,t_n)-u_{i,j}^n||_2 \leq  C (\tau^2+(\Delta x)^4+(\Delta y)^4),
  \end{split}
  \end{equation*}
with
$i=1,2,\ldots,N_x-1;j=1,2,\ldots,N_y-1;\,n=0,1,\ldots,N_t.$
\end{theorem}
\begin{proof}
Taking  $e_{i,j}^n=u(x_i,y_j,t_n)-u_{i,j}^n$, and
 subtracting (\ref{3.9}) from (\ref{3.20}),  we obtain
 \begin{equation}\label{4.6}
  (I-\mathcal{A}_x )(I-\mathcal{A}_y )\mathbf{e}^{n+1}= (I+\mathcal{A}_x )(I+\mathcal{A}_y )\mathbf{e}^{n}+\mathbf{R}^{n+1},
\end{equation}
where $\mathcal{A}_x$ and  $\mathcal{A}_y$ are given in (\ref{3.18}), and
\begin{equation*}
\begin{split}
\mathbf{e}^{n}&=[e_{1,1}^n,e_{2,1}^n,\ldots,e_{N_x-1,1}^n,
   e_{1,2}^n,e_{2,2}^n,\dots,e_{N_x-1,2}^n,\ldots,e_{1,N_y-1}^n,e_{2,N_y-1}^n,\ldots,e_{N_x-1,N_y-1}^n]^T,\\
\mathbf{R}^{n}&=[R_{1,1}^n,R_{2,1}^n,\ldots,R_{N_x-1,1}^n,
   R_{1,2}^n,R_{2,2}^n,\dots,R_{N_x-1,2}^n,\ldots,R_{1,N_y-1}^n,R_{2,N_y-1}^n,\ldots,R_{N_x-1,N_y-1}^n]^T,\\
\end{split}
\end{equation*}
and  $|R_{i,j}^{n+1}|\leq \widetilde{c} \tau(\tau^2+(\Delta x)^4+(\Delta y)^4)$ is given in (\ref{3.10}).

From (\ref{4.4}),  $\mathcal{A}_x$ and $\mathcal{A}_y$ commute, then Eq. (\ref{4.6}) can be rewritten as
 \begin{equation*}
  \mathbf{e}^{n+1}= (I-\mathcal{A}_x )^{-1}(I+\mathcal{A}_x )(I-\mathcal{A}_y )^{-1}(I+\mathcal{A}_y )\mathbf{e}^{n}
  +(I-\mathcal{A}_x )^{-1}(I-\mathcal{A}_y )^{-1}\mathbf{R}^{n+1}.
\end{equation*}
Again, similar to the proof of Theorem 4.2 of [\cite{Deng:11}], we know that $\|(I-\mathcal{A}_\nu)^{-1}(I+\mathcal{A}_\nu)\|_2$  and $\|(I-\mathcal{A}_\nu)^{-1}\|_2$  are less than $1$,  where $\nu=x,y$.
 Then there exists
   \begin{equation*}
  || \mathbf{e}^{n}||_2\leq \sum_{k=0}^{n-1}|\mathbf{R}^{k+1}|\leq
  c (\tau^2+(\Delta x)^4+(\Delta y)^4).
\end{equation*}
\end{proof}

\section{Numerical results}
In this section, we numerically verify the above theoretical results including convergence
rates and numerical stability.  And the $ l_\infty$ norm is used to measure the numerical errors.
\subsection{Numerical results for 1D}\label{Sec4.1}

Consider the one-dimensional fractional diffusion
equation (\ref{3.1}) in the domain $0< x < 2 $, $0 < t \leq 1$, with the variable  coefficients $d_{+}(x)=x^\alpha$, $d_{-}(x)=2x^\alpha$, and the forcing function
\begin{equation*}
\begin{split}
 f(x,t)=&cos(t+1)x^4(2-x)^4 -x^\alpha sin(t+1) \Big[\frac{\Gamma(9)}{\Gamma(9-\alpha)}(x^{8-\alpha}+2(2-x)^{8-\alpha})\\
 &-8\frac{\Gamma(8)}{\Gamma(8-\alpha)}(x^{7-\alpha}+2(2-x)^{7-\alpha})
 +24\frac{\Gamma(7)}{\Gamma(7-\alpha)}(x^{6-\alpha}+2(2-x)^{6-\alpha})\\
 &-32\frac{\Gamma(6)}{\Gamma(6-\alpha)}(x^{5-\alpha}+2(2-x)^{5-\alpha})
 +16\frac{\Gamma(5)}{\Gamma(5-\alpha)}(x^{4-\alpha}+2(2-x)^{4-\alpha}) \Big],
  \end{split}
\end{equation*}
and the initial condition $u(x,0)=sin(1)x^4(2-x)^4$, the boundary conditions $u(0,t)=u(1,t)=0$,
and the exact solution of the equation is $u(x,t)=sin(t+1)x^4(2-x)^4.$

\begin{table}[h]\fontsize{9.5pt}{12pt}\selectfont
 \begin{center}
  \caption {The maximum errors and convergent orders for the scheme (\ref{3.7}) of the one-dimensional  fractional  diffusion equation (\ref{3.1}) at t=1 and $\tau=h^2$.}. \vspace{5pt}
\begin{tabular*}{\linewidth}{@{\extracolsep{\fill}}*{8}{c}}                                    \hline  
$(p,q,r,s,\overline{p},\overline{q},\overline{r},\overline{s})$ &$h$ & $\alpha=1.1$  & Rate        & $\alpha=1.9$ &   Rate    \\\hline
                         &        ~~1/10&  4.7842e-03  &             & 5.8264e-03  &            \\
(1,2,1,0,1,2,1,-2)       &        ~~1/20&  2.5436e-04  &  4.2333     & 5.9999e-04  & 3.2796     \\
                         &        ~~1/40&  1.9662e-05  &  3.6934     & 4.6242e-05  & 3.6977      \\
                         &        ~~1/60&  4.1748e-06  &  3.8218     & 9.7725e-06  & 3.8334       \\\hline 
                         &        ~~1/10&  8.5475e-03  &             & 5.5003e-03  &            \\
(1,2,1,-3,1,2,1,-2)      &        ~~1/20&  4.9722e-04  &  4.1035     & 5.7476e-04  & 3.2585     \\
                         &        ~~1/40&  3.9559e-05  &  3.6518     & 4.4490e-05  & 3.6914      \\
                         &        ~~1/60&  8.6604e-06  &  3.7464     & 9.4148e-06  & 3.8301       \\\hline 
    \end{tabular*}\label{tab:1}
  \end{center}
\end{table}
Table \ref{tab:1} shows the maximum errors, at time $t=1$ with $\tau=h^2$, the numerical results confirm the convergence with the global truncation error $\mathcal{O}(\tau^2+h^4)$.

\subsection{Numerical results for 2D}

Consider the two-dimensional fractional  diffusion
equation (\ref{1.2}), where $0< x < 2 $, $0< y < 2 $,  and $0 < t \leq 1$, with the variable coefficients $d_{+}(x,y)=x^\alpha$, $d_{-}(x,y)=2x^\alpha$,
 and $e_{+}(x,y)=y^\beta$, $e_{-}(x,y)=2y^\beta$, and
the initial condition $u(x,y,0)=sin(1)x^4(2-x)^4y^4(2-y)^4$ with the zero boundary conditions,
and the exact solution of the equation is $$u(x,y,t)=sin(t+1)x^4(2-x)^4y^4(2-y)^4.$$
From the above conditions, it is easy to get the forcing function $f(x,y,t)$.

\begin{table}[h]\fontsize{9.5pt}{12pt}\selectfont
  \begin{center}
  \caption {The maximum errors  and convergent orders for the scheme (\ref{3.20}) of the two-dimensional  fractional  diffusion equation (\ref{1.2}) at t=1 and $\tau=(\Delta x)^2=(\Delta y)^2$.} \vspace{5pt}
  \begin{tabular*}{\linewidth}{@{\extracolsep{\fill}}*{8}{c}}                                    \hline  
$(p,q,r,s,\overline{p},\overline{q},\overline{r},\overline{s})$ &$\Delta x$ & $\alpha=\beta=1.1$  & Rate        & $\alpha=1.8,\beta=1.9$ &   Rate    \\\hline
                         &        ~~1/10&  8.6154e-03  &             & 6.5211e-03  &            \\
(1,2,1,0,1,2,1,-2)       &        ~~1/20&  5.4115e-04  &  3.9928     & 4.4802e-04  & 3.8635     \\
                         &        ~~1/30&  1.2626e-04  &  3.5894     & 8.8416e-05  & 4.0023      \\
                         &        ~~1/40&  4.3328e-05  &  3.7177     & 2.7791e-05  & 4.0229      \\\hline 
                         &        ~~1/10&  1.0110e-02  &             & 6.6368e-03  &            \\
(1,2,1,-3,1,2,1,-2)      &        ~~1/20&  6.3881e-04  &  3.9842     & 4.5471e-04  &  3.8675     \\
                         &        ~~1/30&  1.4363e-04  &  3.6806    & 8.9704e-05  & 4.0032     \\
                         &        ~~1/40&  4.8431e-05  &  3.7788     & 2.8199e-05  & 4.0226       \\\hline 
    \end{tabular*}\label{tab:2}
  \end{center}
\end{table}

Table \ref{tab:2} displays the maximum errors of  the scheme (\ref{3.20}),
and confirms the desired convergence with the global truncation error $\mathcal{O}(\tau^2+(\Delta x)^4+(\Delta y)^4)$.

\section{Conclusions}

Based on the Lubich's operators, this work provides a new idea to obtain the high order discretization schemes for space fractional derivative.  We obtain the effective difference operators with 2nd order, 3rd order, and 4th order accuracy, called WSLD operators. For further checking the efficiency of the high order schemes, we apply the 4th order scheme to solve the space fractional diffusion equation with variable coefficients; and the detailed theoretical analysis and numerical verifications are presented. Hopefully, the higher order (5th order, 6th order, etc.) schemes can be obtained by following the idea given in this paper. In fact, for any fixed convergent order, the obtained difference operators are a class of difference operators, not just one particular operator.

\section*{Acknowledgments} This work was supported by the National Natural Science Foundation of China under
Grant No. 11271173 and the Program for New Century Excellent Talents in University under Grant No.
NCET-09-0438.

\end{document}